\documentclass[10pt,oneside]{article}

\usepackage{lmodern}
\usepackage{times}

\usepackage{amssymb,amsmath,amsthm}
\usepackage{bbold}
\usepackage[short]{optidef}

\usepackage{framed,multirow,multicol}

\usepackage{xcolor,xspace}
\usepackage[normalem]{ulem}
\usepackage{enumerate}
\usepackage{verbatim}
\usepackage{makeidx,latexsym}

\usepackage[numbers,sort&compress,square,comma]{natbib}

\usepackage[english]{babel}
\usepackage{bbm}
\usepackage{pgfplots}

\usepackage{parskip}
\usepackage[letterpaper,margin=1.1in]{geometry}

\usepackage[utf8]{inputenc} %
\usepackage[T1]{fontenc}    %
\usepackage[colorlinks=True,citecolor=blue]{hyperref}       %
\usepackage{url}            %
\usepackage{booktabs}       %
\usepackage{amsfonts}       %
\usepackage{nicefrac}       %
\usepackage{microtype}      %
\usepackage{xcolor}     %
\usepackage{float}
\usepackage{graphicx}
\usepackage{subcaption}

\usepackage{algorithm}
\usepackage{algorithmic}
\usepackage{amsmath} 
\usepackage{amsthm}
\usepackage[capitalize,noabbrev]{cleveref}
\usepackage{bm}
\newcommand{\bmw}{\bm{w}}
\newcommand{\bmwt}{\bm{w_t}}

\newcommand{\bmws}{\bm{w_s}}

\newcommand{\bmp}{\bm{p}}

\newcommand{\bmx}{\bm{x}}
\newcommand{\bmy}{\bm{y}}

\newcommand{\bmomega}{\bm{\omega}}
\newcommand{\bmc}{\bm{x_t}}
\newcommand{\bmn}{\bm{x_{t+1}}}

\newcommand{\bmh}{\bm{\hat x}}
\newcommand{\bmo}{\bm{x^*}}

\newcommand{\grad}{\nabla}

\usepackage{amsmath}
\DeclareMathOperator*{\argmin}{argmin}
\DeclareMathOperator*{\regret}{regret}
\DeclareMathOperator*{\violation}{violation}

\newcommand{\bbE}{\mathbb{E}}
\newcommand{\bbP}{\mathbb{P}}

\newtheorem{assumption}{Assumption}
\newtheorem{definition}{Definition}

\newtheorem{theorem}{Theorem}
\newtheorem{lemma}{Lemma}[section]

\newtheorem{remark}[lemma]{Remark}

\usepackage{bigfoot}

\DeclareNewFootnote{AAffil}[arabic]
\DeclareNewFootnote{ANote}[fnsymbol]

\usepackage{etoolbox}
\makeatletter
\patchcmd\maketitle{\def\@makefnmark{\rlap{\@textsuperscript{\normalfont\@thefnmark}}}}{}{}{}
\makeatother

\makeatletter
\def\thanksAAffil#1{%
  \footnotemarkAAffil\protected@xdef\@thanks{\@thanks%
        \protect\footnotetextAAffil[\the \c@footnoteAAffil]{#1}}%
}
\def\thanksANote#1{%
  \footnotemarkANote%
  \protected@xdef\@thanks{\@thanks%
        \protect\footnotetextANote[\the \c@footnoteANote]{#1}}%
}
\makeatother

\title{Projection-Free Online Convex Optimization with Stochastic Constraints}

\author{
	Duksang Lee
	\thanksAAffil{Department of Industrial and Systems Engineering, KAIST, Daejeon 34141, Republic of Korea}
	\and
	Nam Ho-Nguyen
	\thanksAAffil{Discipline of Business Analytics, The University of Sydney, Sydney, NSW 2006, Australia}
	\and
	Dabeen Lee
	\FootnotemarkAAffil{1}$~$
    \thanksANote{Correspondence to <\url{dabeenl@kaist.ac.kr}>}
}	

\date{\today}

\begin{document}

\maketitle

\begin{abstract}
    This paper develops projection-free algorithms for online convex optimization with stochastic constraints. We design an online primal-dual projection-free framework that can take any projection-free algorithms developed for online convex optimization with no long-term constraint. With this general template, we deduce sublinear regret and constraint violation bounds for various settings. Moreover, for the case where the loss and constraint functions are smooth, we develop a primal-dual conditional gradient method that achieves $O(\sqrt{T})$ regret and $O(T^{3/4})$ constraint violations. Furthermore, for the setting where the loss and constraint functions are stochastic and strong duality holds for the associated offline stochastic optimization problem, we prove that the constraint violation can be reduced to have the same asymptotic growth as the regret.
\end{abstract}

\section{Introduction}

\emph{Online convex optimization (OCO)} is a widely used  framework for decision-making under uncertainty. In OCO, the decision-maker attempts to minimize a sequence of convex loss functions chosen by the adversarial environment. At each iteration, the decision-maker chooses a decision without knowing the loss function, after which the associated loss is revealed by the environment. Based on these repeated interactions, the decision-maker adapts to the environment so as to minimize the total cumulative loss. As the description suggests, the OCO framework is useful for designing iterative solution methods for optimizing a complex system even under limited information.

OCO with \emph{long-term constraints}~\citep{long-term-1} extends the OCO framework. When the problem involves complex functional constraints, projection onto the feasible set can be difficult. For such scenarios, an alternate way is to take and aggregate the constraint functions and then require satisfying the aggregated constraint in the long run. For resource planning, there is a budget for each period, but budgets can be pooled over multiple time periods, over which resource allocation is flexible. The framework is given as the following online optimization model. 
\begin{equation}\label{coco}
    \min_{\bm{x_1},\ldots,\bm{x_T}\in\mathcal{X}}\quad \sum_{t=1}^Tf_t(\bmc)\quad \text{s.t.}\quad \sum_{t=1}^T g_t(\bmc)\leq 0
\end{equation}
where $\{f_t\}_{t=1}^T$ and $\{g_t\}_{t=1}^T$ are the convex loss and constraint functions over $T$ time periods. Here, the decision-maker selects $\bmc$ from a domain $\mathcal{X}$ based on the history up to time step $t$ before observing $f_t$ and $g_t$. The setting where the constraint functions $g_1,\ldots, g_T$ in~\eqref{coco} are independent and identically  distributed (i.i.d.) with an unknown probability distribution is referred to as OCO with \emph{stochastic constraints}~\citep{OCO-stochastic}. 

The existing algorithms for OCO with long-term constraints and stochastic constraints, however, still require projections onto the domain $\mathcal{X}$. The online primal-dual augmented Lagrangian algorithm~\citep{long-term-1,long-term-2} and the drift-plus-penalty algorithm~\citep{long-term-3,yu-neely} provide the two main algorithmic frameworks for OCO with long-term constraints, and both are variants of the primal-dual projected gradient method, requiring projection onto the domain at each iteration. When $\mathcal{X}$ is given by linear inequalities, projection onto it boils down to solving a quadratic program in the general case. Matrix completion for recommender systems requires projection onto a spectahedron~\citep{hazan-projection-free}. 

There has been a surge of interest in projection-free algorithms based on the famous Frank-Wolfe method, replacing each projection step by a linear optimization over $\mathcal{X}$~\citep{hazan-projection-free}. \cite{conservative} developed a projection-free algorithm for stochastic optimization with stochastic constraints. However, as far as we know, no projection-free algorithm exists for OCO with long-term constraints. Motivated by this, we develop projection-free algorithms for online convex optimization with stochastic constraints. 

\paragraph{Our Contributions}
\begin{enumerate}
    \item We design an online primal-dual projection-free learning framework for OCO with stochastic constraints. The framework works as a general template, and it can take any projection-free algorithm developed for OCO with no long-term constraint. In particular, we apply the framework to different settings depending on (1) whether the loss and constraint functions are smooth or not and (2) whether the loss functions are arbitrary or stochastic. We provide sublinear regret and constraint violation bounds for various settings.

    \item For the case where the loss and constraint functions are smooth, we develop Primal-Dual Meta-Frank-Wolfe, which is a variant of the Meta-Frank-Wolfe algorithm of~\citep{meta-frank-wolfe}. Subject to a per-iteration cost of $O(\sqrt{T})$, the algorithm guarantees $O(\sqrt{T})$ regret and $O(T^{3/4})$ constraint violation.

    \item When the loss functions are also stochastic and strong duality of the associated offline stochastic optimization problem is satisfied, we prove that the algorithms achieve smaller constraint violations. More precisely, we may reduce the constraint violation  to have the same asymptotic growth as the regret if the strong duality assumption holds and the loss functions are i.i.d. with a probability distribution.
\end{enumerate}
Our results are summarized in \Cref{results1}.
\begin{table}        
\begin{center}
\begin{tabular}{c|c|c|c|c|c}
\toprule
& {Functions} &{Setting} & {Regret}  &  Constraint violation& Per-round cost\\
\hline
\multirow{5}{*}{Alg. \ref{alg1}} & Non-smooth  & Adversarial & $T^{5/6+\beta}$ & $T^{11/12-\beta/2}$ &  1 \\
& Non-smooth  & Stochastic, SD & $T^{5/6}$ & $T^{5/6}$ &  1 \\
& Smooth & Adversarial & $T^{4/5+\beta}$ & $T^{9/10-\beta/2}$ &  1\\
& Smooth & Stochastic & $T^{3/4+\beta}$ & $T^{7/8-\beta/2}$ &  1\\
& Smooth & Stochastic, SD & $T^{3/4}$ & $T^{3/4}$ &  1\\
\midrule
\multirow{2}{*}{Alg. \ref{alg1}} & Non-smooth & Adversarial &  $T^{3/4+\beta}$ & $T^{7/8-\beta/2}$ &  $T$ \\
& Non-smooth & Stochastic, SD &  $T^{3/4}$ & $T^{3/4}$ &  $T$ \\
\midrule
\multirow{2}{*}{Alg. \ref{alg2}} & Smooth & Adversarial & $T^{1/2+\beta}$ & $T^{3/4-\beta/2}$ &  $\sqrt{T}$ \\
& Smooth & Stochastic, SD & $T^{1/2}$ & $T^{1/2}$ &  $\sqrt{T}$ \\
\bottomrule
\end{tabular}
\end{center}
\caption{Regret and constraint violation bounds for various settings. Here SD means strong duality holds for the associated stochastic optimization problem; see \cref{sec:stochastic-loss}.}\label{results1}
\end{table}
In column ``Setting", ``Adversarial" means that the loss functions are adversarially chosen as the standard OCO framework, and ``Stochastic" means that the loss functions are i.i.d. with an unknown probability distribution.

\section{Related Work}

Our work is related to the literature on online convex optimization with long-term constraints and the projection-free online learning and optimization literature.

\paragraph{OCO with Long-Term Constraints} The most general setting is where $\{g_t\}_{t=1}^T$ is adversarially chosen and the benchmark $\bmo$ is set to an optimal fixed solution of~\eqref{coco}. However, for the general setting, it is known that it is impossible to simultaneously bound the regret and constraint violation by a sublinear function in $T$~\citep{mannor-coco}. \citet{long-term-1} then considered the special case where $g_t=g$ for some fixed function $g$ for all $t\in[T]$, and they provided an augmented-Lagrangian-based algorithm that achieves $O(\sqrt{T})$ regret and $O(T^{3/4})$ constraint violation. \citet{long-term-2} modified the algorithm to obtain $O(T^{\max\{\beta,1-\beta\}})$ regret and $O(T^{1-\beta/2})$ constraint violation where $\beta\in(0,1)$ is an algorithm parameter. 

\citet{OCO-stochastic} considered the case where $g_t$'s are time-varying but are i.i.d. with an unknown probability distribution, {i.e., stochastic constraints}. They provided the drift-plus-penalty (DPP) algorithm which guarantees $O(\sqrt{T})$ expected regret and $O(\sqrt{T})$ expected constraint violation under Slater's condition. \citet{dpp-md} gave a variant of DPP attaining the same asymptotic performance {under a more general strong duality assumption.}

\citet{yu-neely,cautious,stronger-benchmark} consider adversarially chosen constraint functions but they set different benchmarks with more restrictions. Recently, \citet{castiglioni2022a} came up with a unifying framework that works for both stochastic and adversarial constraint functions. \citet{cumulative-1,cumulative-2,cumulative-3} studied the notion of \emph{cumulative constraint violation}, given by $\sum_{t=1}^T\left[g_t(\bmc)\right]_+$ where $\left[a\right]_+=\max\{a,0\}$ over $a\in\mathbb{R}$, instead of imposing long-term constraints.

\paragraph{Projection-Free Online Learning and Optimization} \citet{hazan-projection-free} developed the first online projection-free algorithm, based on the Frank-Wolfe method~\citep{frank-wolfe}, that guarantees $O(T^{3/4})$ regret for non-smooth loss functions. \citet{Garber-Hazan} provided an improved algorithm for smooth and strongly convex loss functions. After these works, there has been many results on projection-free algorithms for online convex optimization with no long-term constraint, and \Cref{known-bounds} summarizes the known regret bounds for various settings.
\begin{table}        
\begin{center}
\begin{tabular}{c|c|c|c|c}
\toprule
{Loss function} &{Setting} & {Regret}  & Per-round cost& {Reference} \\
\hline
Non-smooth  & Adversarial & $T^{3/4}$ & 1 & OCG  \citep{hazan-oco}\\
Smooth & Adversarial & $T^{2/3}$ & 1 & OSPF \citep{faster-projection-free}\\
Smooth & Stochastic & $T^{1/2}$ & 1 & ORGFW \citep{Xie_Shen_Zhang_Wang_Qian_2020}\\
\midrule
{Non-smooth} & {Adversarial} & {$T^{1/2}$} & {$T$} & SFTRL {\citep{faster-projection-free}}\\
Smooth & Adversarial & $T^{1/2}$ & $\sqrt{T}$ & MFW  \citep{meta-frank-wolfe}\\
\bottomrule
\end{tabular}
\end{center}
\caption{Regret bounds for projection-free algorithms for OCO with no long-term constraint.}\label{known-bounds}
\end{table}

\section{Online Convex Optimization with Stochastic Constraints}\label{sec:oco}

Let $\mathcal{X}\subseteq\mathbb{R}^d$ be a known fixed compact convex set. Let $f_1,\ldots, f_T:\mathbb{R}^d\to\mathbb{R}$ be a sequence of arbitrary convex loss functions. Let $\bar g(\bmx)=\mathbb{E}_{\bmomega}\left[g(\bmx,\bmomega)\right]:\mathbb{R}^d\to\mathbb{R}$ be a function where $g(\bmx,\bmomega)$ is convex with respect to $\bmx\in\mathcal{X}$ and the expectation is taken with $\bmomega\in\Omega$ from an unknown distribution.

Like~\citet{OCO-stochastic}, we take the benchmark decision $\bmo$ defined as an optimal solution to 
\begin{equation}\label{problem}
    \min_{\bmx\in\mathcal{X}}\ \ \sum_{t=1}^Tf_t(\bmx)\ \ \text{s.t.}\quad \bar g(\bmx) = \mathbb{E}_{\bmomega}\left[g(\bmx,\bmomega)\right] \leq 0.
\end{equation}
{Instead of direct access to $\bar{g}$, we are presented constraint functions $g_1,\ldots, g_T:\mathbb{R}^d\to\mathbb{R}$ where $g_t(\bmx):=g(\bmx, \bm{\omega_t})$ for $t\in[T]$ where $\bm{\omega_1},\ldots, \bm{\omega_T}$ are i.i.d. samples of $\bmomega$.}

Our goal is to design an algorithm {for choosing $\bmc$, $t \in [T]$} that guarantees a sublinear regret against the benchmark $\bmo$ and a sublinear constraint violation at the same time, {under the condition that each $\bmc$ only depends on functions from previous time steps $f_1,\ldots,f_{t-1},g_1,\ldots,g_{t-1}$}. Here, the regret and constraint violation are defined as follows:
$$\regret(T)=\sum_{t=1}^Tf_t(\bmc)-\sum_{t=1}^Tf_t(\bmo),\quad \violation(T)=\sum_{t=1}^Tg_t(\bmc).$$
We focus on the single constraint setting for simplicity, but our framework easily extends to multiple constraints.
We assume that each loss $f_t$ is chosen adversarially, but is independent of $\bm{\omega_s}$ for $s\geq t+1$. {In other words, $f_t$ can be chosen with full knowledge of the history up to time $t$, but \emph{not} future random realizations.}

We also consider the case when the loss functions are stochastic i.i.d. realizations of some random function $f(\bmx,\bmomega)$, i.e., $f_t$ is given by $f_t(\bmx)=f(\bmx,\bm{\omega_t})$ where $f(\bmx, \bmomega)$ is convex with respect to $\bmx\in\mathcal{X}$. {Note here that the random variable $\bmomega$ is the same as the one that appears in $g_t$, meaning that $f_t$ and $g_t$ are possibly dependent.} A direct application of this setting is \emph{stochastic constrained stochastic optimization}, that is formulated as the following optimization problem.
$$\min_{\bmx\in\mathcal{X}}\quad \bar f(\bmx)=\mathbb{E}_{\bmomega}\left[f(\bmx,\bmomega)\right]\quad \text{s.t.}\quad \bar g(\bmx)\leq 0.$$
An iterative algorithm would obtain an i.i.d. sample $\bm{\omega_t}$ of $\bmomega$ at each iteration $t$ and consider $f_t=f_t(\cdot,\bm{\omega_t})$ and $g_t=g_t(\cdot,\bm{\omega_t})$. Given $\{\bmc\}_{t=1}^T$, we may obtain $\bm{\bar x_T} = (1/T)\sum_{t=1}^T\bmc$. By Jensen's inequality, the optimality gap and constraint violation of $\bm{\bar x_T}$ are given by 
$$\mathbb{E}\left[\bar f(\bm{\bar x_T})\right]-\bar f(\bmx)\leq \frac{1}{T}\mathbb{E}\left[\sum_{t=1}^Tf_t(\bmc) - \sum_{t=1}^T f_t(\bmx)\right],\quad \bar g(\bm{\bar x_T})\leq\frac{1}{T}\mathbb{E}\left[\sum_{t=1}^T g_t(\bmc)\right].$$

\section{Online Primal-Dual Projection-Free Learning Framework}\label{sec:alg1}

\cref{alg1} provides a projection-free algorithmic framework for online convex optimization with stochastic constraints. The algorithm is a general template that can take any projection-free algorithm for online convex optimization with no long-term constraint.
{This allows us to use different projection-free algorithms depending on the structure of the loss and constraint functions.}

\begin{algorithm}[tb]
    \caption{Online Primal-Dual Projection-Free Learning Framework}
    \label{alg1}
    \begin{algorithmic}
        \STATE {\bfseries Initialize:} time horizon $T$, number of blocks $Q$, block size $K$, initial iterates $\bm{x_1}\in\mathcal{X}$, $\lambda_1=0$, the number of blocks $Q$, step size $\mu$, augmentation parameter $\theta$, and a projection-free convex optimization oracle $\mathcal{E}$.
        
        \FOR{$q=1$ {\bfseries to} $Q$}
        
        \FOR{$k=1$ {\bfseries to} $K$}
        \STATE Observe $f_t$ and $g_t$.

        \STATE Use oracle $\mathcal{E}$ to obtain $\bmn$ based on $h_t(\bmx)=f_t(\bmx)+\lambda_qg_t(\bmx)$ and $\bmx=\bmc$. 

        \STATE Set $t\leftarrow t+1$.

        \ENDFOR
        \STATE Set 
        $\lambda_{q+1}=\left[(1-\theta\mu)\lambda_{q} +\mu\sum_{t=(q-1)K+1}^{qK}g_t(\bmc) \right]_+$.
        \ENDFOR
    \end{algorithmic}
\end{algorithm}

The idea behind~\cref{alg1} is as follows. First, we break the time horizon $T$ into $Q$ blocks. Each block has $K$ time steps. Then for each of the $Q$ blocks, we use as an oracle a projection-free algorithm developed for online convex optimization with no long-term constraint. To be specific, for block $q\in [Q]$, the oracle is applied to 
$$f_t(\bmx)+\lambda_q g_t(\bmx)$$
which is the loss function $f_t$ at time $t$ penalized by the constraint function $g_t$ where $\lambda_q$ is the penalty parameter for block $q$. Once iterations in block $q$ are completed, we update the penalty parameter $\lambda_q$ (the dual variable) based on the constraint function values realized in block $q$. The update rule is motivated by the online primal-dual augmented Lagrangian algorithm due to \citet{long-term-1,long-term-2},
{
and it makes use of the following augmented Lagrangian function
$$L_t(\bmx, \lambda) = f_t(\bmx)+ \lambda g_t(\bmx) - \frac{\theta}{2K}\lambda^2.$$
Given functions $f_t,g_t$ for time steps $t=(q-1)K+1,\ldots,qK$ observed in block $q$, \cref{alg1} then updates the penalty parameter using the gradient ascent-type update
$$\lambda_{q+1} = \left[\lambda_q + \mu \sum_{t=(q-1)K+1}^{qK}\grad_\lambda L_t(\bmc,\lambda_q)\right]_+ = \left[(1-\theta \mu) \lambda_q + \mu \sum_{t=(q-1)K+1}^{qK} g_t(\bmc) \right]_+,$$
where $\mu$ is a step size.
}

Throughout this section, we work over the $\ell_2$ norm $\|\cdot\|_2$ in $\mathbb{R}^d$ for simplicity. 
\begin{definition}[Lipschitz continuity]
We say that a function $h:\mathcal{X}\to\mathbb{R}$ is $D$-Lipschitz for some $D\geq 0$ if $\|\grad h(\bmx)\|_2\leq D$ for all $\bmx\in\mathcal{X}$. 
\end{definition}

\begin{definition}[Smoothness]
We say that a function $h:\mathcal{X}\to\mathbb{R}$ is $L$-smooth for some $L\geq 0$ if $\|\grad h(\bmx)-\grad h(\bmy)\|_2\leq L\|\bmx-\bmy\|_2$ for all $\bmx,\bmy\in\mathcal{X}$. 
\end{definition}

We formally define the notion of oracle $\mathcal{E}$ used as a subroutine in~\cref{alg1}.

\begin{definition}[Oracle]
We say that $\mathcal{E}$ is an $(\alpha,C_0,C_1,C_2)$-oracle for some $\alpha\in(0,1)$ and $C_0,C_1,C_2\geq0$ if for any sequence of (adversarial or stochastic) convex loss functions $h_1,\ldots, h_K$ that are $D$-Lipschitz and $L$-smooth over domain $\mathcal{X}\subseteq\mathbb{R}^d$, oracle $\mathcal{E}$ guarantees
$$\mathbb{E}\left[\sum_{k=1}^K h_k(\bm{x_k}) -\min_{x\in \mathcal{X}} \sum_{k=1}^K h_k(\bm{x})\right]\leq (C_0+C_1D +C_2L)K^\alpha$$
where the expectation is taken over the randomness of oracle $\mathcal{E}$ itself and the randomness of the convex loss functions. Constants $C_0,C_1,C_2$ are independent of parameters $D,L,K$. 
\end{definition}

\begin{assumption}
When a function is non-smooth, we assume that it is $\infty$-smooth. Moreover, we assume that $0\cdot \infty = 0$.
\end{assumption}
\begin{remark}
Let $\mathcal{E}$ be an $(\alpha,C_0,C_1,C_2)$-oracle for some $\alpha\in(0,1)$ and $C_0,C_1,C_2\geq0$. If $C_2=0$, then $\mathcal{E}$ guarantees an $O(K^\alpha)$ regret for any Lipschitz  loss functions that can be non-smooth. If $C_2>0$, then $\mathcal{E}$ guarantees an $O(K^\alpha)$ regret only if the loss functions are smooth.
\end{remark}

If we are given an $(\alpha, C_0,C_1,C_2)$-oracle, then we set the parameters of \cref{alg1} as follows:
\begin{align}\label{alg1-parameters}
\begin{aligned}
Q=T^\frac{2-2\alpha}{3-2\alpha}, K=T^\frac{1}{3-2\alpha}, \beta \in \left[0, \frac{1-\alpha}{3-2\alpha} \right],  \theta = 3(C_1D+C_2L) T^{\frac{\alpha}{3-2\alpha}-\beta},  \mu=\frac{1}{\theta(Q+1)}.
\end{aligned}
\end{align}
We assume that both $T^\frac{2-2\alpha}{3-2\alpha}$ and $T^\frac{1}{3-2\alpha}$ are integers. Even if they are not, we {can take the ceiling $\lceil \cdot \rceil$ as needed}, and our framework still achieves the same asymptotic guarantees.
\begin{theorem}\label{thm:alg1}
Suppose that loss functions $f_1,\ldots,f_T$ and stochastic constraint functions $g_1,\ldots, g_T$ are $D$-Lipschitz and $L$-smooth.
If $\mathcal{E}$ is an $(\alpha,C_0,C_1,C_2)$-oracle for some $\alpha\in(0,1)$ and $C_0,C_1,C_2\geq0$, then \cref{alg1} whose parameters are set as in~\eqref{alg1-parameters} guarantees that
$$\mathbb{E}\left[\sum_{t=1}^T f_t(\bm{x_t}) -\min_{x\in \mathcal{X}} \sum_{t=1}^T f_t(\bm{x})\right] =O\left(T^{\frac{2-\alpha}{3-2\alpha}+\beta}\right),\quad \mathbb{E}\left[\sum_{t=1}^T g_t(\bm{x_t})\right] =O\left(T^{\frac{5-3\alpha}{6-4\alpha}-\frac{\beta}{2}}\right)$$
where the expectations are taken over the randomness of oracle $\mathcal{E}$ and the randomness of the loss and constraint functions.
\end{theorem}
In particular, based on projection-free algorithms for OCO with no long-term constraint as in \Cref{known-bounds}, we deduce results in \Cref{results1}. The setting where functions are smooth and $O(\sqrt{T})$ per-round cost is allowed and the settings under strong duality are considered in the next sections.

\section{Primal-Dual Meta-Frank-Wolfe for Smooth Functions}

For the setting where the loss and constraint functions are smooth, \Cref{alg1} guarantees $O(T^{3/4+\beta})$ regret and $O(T^{7/8-\beta/2})$ constraint violation.
In this section, we develop \Cref{alg2} which provides $O(T^{1/2+\beta})$ regret and $O(T^{3/4-\beta/2})$ constraint violation, where we use $O(\sqrt{T})$ gradient evaluations per time step.

\cref{alg2} is a combination of Meta-Frank-Wolfe~\citep{meta-frank-wolfe} for projection-free online convex optimization (with no long-term constraint) and the online primal-dual gradient method~\citep{long-term-1,long-term-2}. We refer to \cref{alg2} as Primal-Dual Meta-Frank-Wolfe (PDMFW). In constrast to \Cref{alg1} that updates the dual variable only when a new block starts, \Cref{alg2} updates the dual variable for every time step. 

\begin{algorithm}[tb]
\caption{Primal-Dual Meta-Frank-Wolfe (PDMFW)}
    \label{alg2}
\begin{algorithmic}
\STATE {\bfseries Initialize:} initial iterates $\bm{x_1}\in\mathcal{X}$, $\lambda_1=0$, step size $\mu$, augmentation parameter $\theta$, inner loop length $K$, projection-free linear optimization oracles $\mathcal{E}^1,\ldots, \mathcal{E}^K$, and step sizes $\gamma_1,\ldots, \gamma_K$.
\FOR{$t=1$ {\bfseries to} $T$}
    \STATE Observe $f_t$ and $g_t$.
    \STATE {\bfseries Primal update:} obtain $\bm{x_{t+1}}$ as follows
    \STATE Set $\bm{x_{t+1}^1} = \bm{x_1}$
    \FOR{$k=1$ {\bfseries to} $K$}
        \STATE Use oracle $\mathcal{E}^k$ (\cref{ftpl}) to obtain $\bm{v_{t+1}^k}$.
        
        \STATE Update 
        $$\bm{x_{t+1}^{k+1}} = \bm{x_{t+1}^k} + \gamma_k\left(\bm{v_{t+1}^k}-\bm{x_{t+1}^k}\right)$$
    \ENDFOR
    \STATE Set $$\bm{x_{t+1}}=\bm{x_{t+1}^{K+1}}$$
    \STATE {\bfseries Dual update:}
    $$\lambda_{t+1}=\left[(1-\theta\mu)\lambda_{t} +\mu g_t(\bmc) \right]_+$$
\ENDFOR
\end{algorithmic}
\end{algorithm}
As in~\citep{long-term-1,long-term-2}, PDMFW works over the following \emph{augmented Lagrangian function}. Upon observing $f_t$ and $g_t$ at time $t$, we take 
\begin{equation}\label{lagrangian}
    L_t(\bmx,\lambda)=f_t(\bmx) + \lambda g_t(\bmx)-\frac{\theta}{2}\lambda^2\end{equation}
to compute the next iterate $\bm{x_{t+1}}$. %
At a high level, PDMFW is an online primal-dual framework based on the augmented Lagrangian function~\eqref{lagrangian} that applies a Frank-Wolfe subroutine for the primal update and gradient ascent for the dual update.

The Frank-Wolfe subroutine replaces the projection-based primal update of the online primal-dual gradient method. Starting from $\bm{x_{t+1}^1}=\bm{x_t}$ at time $t$, the Frank-Wolfe procedure runs with $K$ steps and generates $\bm{x_{t+1}^2},\ldots,\bm{x_{t+1}^{K+1}}$. Then we set $\bmn=\bm{x_{t+1}^{K+1}}$. The Frank-Wolfe update at each step $k$ is given by
$\bm{x_{t+1}^{k+1}} = \bm{x_{t+1}^k} + \gamma_k\left(\bm{v_{t+1}^k}-\bm{x_{t+1}^k}\right)$
for some step size $\gamma_k$. Here, the direction $\bm{v_{t+1}^k}$ {would ideally} be a vector $\bm{v}\in\mathcal{X}$ minimizing
$$\grad_{\bmx}L_{t+1}(\bm{x_{t+1}^k},\lambda_{t+1})^\top\bm{v}=\left(\grad f_{t+1}(\bm{x_{t+1}^k})+\lambda_{t+1}\grad g_{t+1}(\bm{x_{t+1}^k})\right)^\top \bm{v}.$$
However, as functions $f_{t+1}$ and $g_{t+1}$ are {only revealed after choosing $\bm{v_{t+1}^k}$, we use a projection-free online linear optimization oracle $\mathcal{E}^k$ to obtain $\bm{v_{t+1}^k}$} based on the history up to $t$.
This idea was first introduced by~\citet{meta-frank-wolfe-0}. For $k=1,\ldots, K$, we set 
$$\gamma_k = \frac{2}{k+1}.$$

For the dual update, we follow the update rule of the online primal-dual gradient method, that is,
$$\lambda_{t+1} =\left[ \lambda_t + \mu\grad_\lambda L_t(\bm{x_t^k},\lambda_t)\right]_+ = \left[(1-\theta\mu)\lambda_{t} +\mu g_t(\bmc) \right]_+$$
where $\mu$ is a step size. For a fixed $\beta\in(0,1/2)$, we set the parameters as follows:
\begin{equation}\label{alg2-parameters}
K=\lfloor T^{\frac{1}{2}+\beta}\rfloor,\quad \theta = \frac{12RD\sqrt{d}}{T^{\frac{1}{2}+\beta}},\quad \mu= \frac{1}{\theta(T+2)}.
\end{equation}
Here, we may set any value between 0 and 1 for $\beta$. We will show that the (expected) regret of PDMFW is $O(T^{\frac{1}{2}+\beta})$ and the (expected) constraint violation is $O(T^{\frac{3}{4}-\frac{\beta}{2}})$. 

For a projection-free online linear optimization oracle, we use the \emph{Follow-The-Perturbed-Leader (FTPL)} algorithm~\citep{Hannan,ftpl}, given as in~\Cref{ftpl}.

\begin{algorithm}[tb]
\caption{Follow-The-Perturbed-Leader for $\mathcal{E}^k$}
\label{ftpl}
\begin{algorithmic}
\STATE Choose a random perturbation vector $\bmp$ from $[0,\delta]^d$ uniformly at random where
$$\delta = \frac{1}{2D\sqrt{d} T^{\frac{1}{2}+\beta}}$$
\FOR{$t=1$ {\bfseries to} $T$}
\STATE Choose 
\[ \bm{v_{t+1}^k} \in \argmin_{\bm{v} \in \mathcal{X}} \left\{ \bmp^\top \bm{v} + \sum_{s=1}^{t-1} h_s^k(\bm{v}) \right\} \]
where
$$h_s^k(\bm{v}) := \left(\grad f_s(\bm{x_s^k}) +\lambda_s \grad g_s(\bm{x_s^k})\right)^\top \bm{v}.$$
\ENDFOR
\end{algorithmic}
\end{algorithm}
The regret of FTPL for online linear optimization has a dependence on $T$ bounded above by $O(\sqrt{T})$~\citep{ftpl}. However, the coefficient $\grad_{\bmx}L_t(\bmc,\lambda_t)=\grad f_t(\bm{x_t^k})+\lambda_t\grad g_t(\bm{x_t^k})$ for time $t$ has dual variable $\lambda_t$, so the regret grows as a function of $\lambda_1,\ldots,\lambda_T$. Therefore, we need a refined regret analysis of FTPL to show how the regret grows as a function of the dual variables $\lambda_1,\ldots,\lambda_T$.

We remark that \cref{alg2} is similar to the OSPHG algorithm of~\citet{Sadeghi1} developed for online DR-submodular maximization, {but modified to be projection-free}.

Throughout this section, we work over the $\ell_1$ norm $\|\cdot\|_1$ in $\mathbb{R}^d$ and its dual $\|\cdot\|_\infty$, the $\ell_\infty$ norm.
\begin{assumption}[Basic assumptions]\label{basic-assumption}
    There are positive constants $D, G, R$ satisfying the following.
    \begin{itemize}
        \item $\|\grad f_t(\bmx)\|_\infty,\|\grad g_t(\bmx)\|_\infty\leq D$ and $g_t(\bmx)\leq G$ for all $t\in[T]$ and $\bmx\in\mathcal{X}$.
        \item $\|\bmx-\bmy\|_1\leq R$ for all $\bmx, \bmy\in\mathcal{X}$.
    \end{itemize}
\end{assumption}

\begin{assumption}[Smoothness]\label{smoothness} There exists a positive constant $L$ such that $$\left\|\grad f_t(\bmx)-\grad f_t(\bmy)\right\|_\infty, \left\|\grad g_t(\bmx)-\grad g_t(\bmy)\right\|_\infty\leq L\left\|\bmx - \bmy\right\|_1$$ for all $t\in[T]$ and $\bmx,\bmy\in\mathcal{X}$.
\end{assumption}

We first provide an upper bound on the expected regret of FTPL (\cref{ftpl}).

\begin{lemma}\label{regret:ftpl}
For each $k=1,\ldots, K$, the expected regret of FTPL (\cref{ftpl}) under linear functions $h_1^k(\bm{v}),\ldots, h_T^k(\bm{v})$ (defined in \cref{ftpl}) is bounded above by 
$$\mathcal{R}^{\mathcal{E}}(T):=RD\sqrt{d} \left(3T^{\frac{1}{2}+\beta} + \frac{1}{T^{\frac{1}{2}+\beta}}\mathbb{E}\left[\sum_{t=1}^T\lambda_t^2\right]\right).$$
\end{lemma}

\Cref{cmfwg:regret} gives bounds on the expected regret and the expected long-term constraint violation under~\cref{alg2}, respectively. Let constants $C_1,C_2$ be defined as
\begin{align*}
        C_1&= 4RD+5R^2L +\frac{R(4D+5RL)^2}{4D\sqrt{d}}+3RD\sqrt{d}+\frac{G^2}{12RD\sqrt{d}},\\
C_2&=\sqrt{96RD\sqrt{d}\left(RD+C_1\right)}.
\end{align*}

\begin{theorem}\label{cmfwg:regret}
Suppose that the loss functions $f_1,\ldots, f_T$ and the stochastic constraint functions $g_1,\ldots, g_T$ satisfy Assumptions~\ref{basic-assumption} and~\ref{smoothness}. Then \cref{alg2} with parameters set according to~\eqref{alg2-parameters} guarantees that
$$\mathbb{E}\left[\sum_{t=1}^T f_t(\bm{x_t}) -\min_{x\in \mathcal{X}} \sum_{t=1}^T f_t(\bm{x})\right] \leq C_1 T^{\frac{1}{2}+\beta},\quad \mathbb{E}\left[\sum_{t=1}^T g_t(\bm{x_t})\right] \leq C_2 T^{\frac{3}{4}-\frac{\beta}{2}}$$
where the expectations are taken over the randomness of $\mathcal{E}^1,\ldots,\mathcal{E}^K$ and the randomness of the loss and constraint functions. 
\end{theorem}

\section{Stochastic Loss Functions under Strong Duality}\label{sec:stochastic-loss}

{In this section, we focus on stochastic loss functions, i.e., $f_t(\bmx)=f(\bmx,\bm{\omega_t})$ in addition to stochastic constraint functions $g_t(\bmx) = g(\bmx,\bm{\omega_t})$. Recall that $\bar f(\bmx) = \mathbb{E}\left[f(\bmx,\bmomega)\right]$ and $\bar g(\bmx) = \mathbb{E}\left[g(\bmx,\bmomega)\right]$. We show that we can obtain improved bounds under strong Lagrangian duality of the following optimization problem:
\begin{equation}\label{conservative-opt}
\min_{\bmx\in\mathcal{X}}\left\{\bar f(\bmx):\ \bar g(\bmx)\leq 0\right\}.
\end{equation}
The Lagrangian of this is
$$\bar L(\bmx,\lambda) :=  \bar f(\bmx)  + \lambda \bar g(\bmx)$$
where $\bmx \in \mathcal{X}$ and $\lambda \geq 0$, and we assume that there exist $\bmo \in \mathcal{X}$ and $\lambda^* \geq 0$ such that
\begin{equation}\label{saddle-point}
\bar L(\bmo, \lambda) \leq \bar L(\bmo, \lambda^*)\leq \bar L(\bmx, \lambda^*), \quad \forall \bmx \in \mathcal{X}, \lambda \geq 0,
\end{equation}
i.e., strong duality holds for \eqref{conservative-opt}. This is satisfied, for example, under Slater constraint qualification, when there exists $\bmh\in\mathcal{X}$ with $\bar g(\bmh)=\mathbb{E}\left[g(\bmh,\bm{\omega})\right]<0$ (see \citet[Proposition 5.1.6]{Bertsekas-nonlinear}), but our analysis allows for more general settings where strong duality holds but Slater constraint qualification may not hold.
}

\begin{lemma}\label{expectation0}
Let $\bm{x_1},\ldots, \bm{x_T}$ and $\lambda_1,\ldots, \lambda_Q$ be the decisions and dual variables chosen by~\cref{alg1}. When \eqref{saddle-point} holds we have
\begin{align*}\mathbb{E}\left[\sum_{t=1}^T g_t(\bmc) \right]&\leq \mathbb{E}\left[\sum_{q=1}^Q\sum_{k=1}^K\left(f_{(q-1)K+k}(\bm{x_{(q-1)K+k}})+(\lambda^*+1)g_{(q-1)K+k}(\bm{x_{(q-1)K+k}})\right)\right]\\
&\quad -\mathbb{E}\left[ \sum_{q=1}^Q\sum_{k=1}^K\left(f_{(q-1)K+k}(\bmo)+\lambda_qg_{(q-1)K+k}(\bmo)\right)\right].\end{align*}
\end{lemma}

\begin{lemma}\label{expectation2}
Let $\bm{x_1},\ldots, \bm{x_T}$ and $\lambda_1,\ldots, \lambda_T$ be the decisions and dual variables chosen by~\cref{alg2}. When \eqref{saddle-point} holds we have
$$\mathbb{E}\left[\sum_{t=1}^T g_t(\bmc) \right]\leq \mathbb{E}\left[\sum_{t=1}^T\left(f_t(\bmc)+(\lambda^*+1)g_t(\bmc)\right) - \sum_{t=1}^T\left(f_t(\bmo)+\lambda_tg_t(\bmo)\right)\right].$$
\end{lemma}
Using these bounds, we deduce the following results. 

\begin{theorem}\label{thm:Slater1}
Suppose that the stochastic loss functions $f_1,\ldots,f_T$ and stochastic constraint functions $g_1,\ldots, g_T$ are $D$-Lipschitz and $L$-smooth. Furthermore, \eqref{saddle-point} is satisfied.
If $\mathcal{E}$ is an $(\alpha,C_0,C_1,C_2)$-oracle for some $\alpha\in(0,1)$ and $C_0,C_1,C_2\geq0$, then \cref{alg1} whose parameters are set as in~\eqref{alg1-parameters} with $\beta=0$ guarantees that
$$\mathbb{E}\left[\sum_{t=1}^T f_t(\bm{x_t}) -\min_{x\in \mathcal{X}} \sum_{t=1}^T f_t(\bm{x})\right] =O\left(T^{\frac{2-\alpha}{3-2\alpha}}\right),\quad \mathbb{E}\left[\sum_{t=1}^T g_t(\bm{x_t})\right] =O\left(T^{\frac{2-\alpha}{3-2\alpha}}\right)$$
where the expectations are taken over the randomness of oracle $\mathcal{E}$ and the randomness of the loss and constraint functions.
\end{theorem}

\begin{theorem}\label{thm:Slater2}
Suppose that the stochastic loss functions $f_1,\ldots, f_T$ and the stochastic constraint functions $g_1,\ldots, g_T$ satisfy Assumptions~\ref{basic-assumption} and~\ref{smoothness}. Furthermore, \eqref{saddle-point} is satisfied. Then
    \cref{alg2} whose parameters are set as in~\eqref{alg2-parameters} with $\beta=0$ guarantees that
$$\mathbb{E}\left[\sum_{t=1}^T f_t(\bm{x_t}) -\min_{x\in \mathcal{X}} \sum_{t=1}^T f_t(\bm{x})\right] \leq C_1 T^{\frac{1}{2}},\quad \mathbb{E}\left[\sum_{t=1}^T g_t(\bm{x_t})\right] \leq C_2 T^{\frac{1}{2}}$$
where the expectations are taken over the randomness of $\mathcal{E}^1,\ldots,\mathcal{E}^K$ and the randomness of the loss and constraint functions. 
\end{theorem}

\section{Numerical Experiments}
In this section, we present our experimental results to test the numerical performance of our projection-free algorithms, Algorithms~\ref{alg1} and~\ref{alg2}, for online convex optimization with stochastic constraints. For Algorithm~\ref{alg1}, we use the algorithms listed in~\Cref{known-bounds}. 
We consider the online matrix completion problem as in~\citep{hazan-oco,meta-frank-wolfe}. For our experiments, we generate instances with synthetic simulated data. 
Here, to test our framework, we impose stochastic constraints. 

We are given an $m\times n$ matrix $M$, and at each iteration $t$, we observe a subset $B_t\subseteq \{(i,j):1\leq i\leq m, 1\leq j\leq n\}$ of the entries of $M$. Here, $M$ may encode the preferences of users over certain media items, in which case, $B_t$ corresponds to the ratings inputted by some users at time $t$. Therefore, based on the sequence of subsets $B_1,\ldots, B_T$ that we observe over time, we want to infer the underlying matrix $M$. To be specific, we consider the following problem formulation: 
\begin{align*}
\min\quad \frac{1}{2}\sum_{t=1}^{T}\sum_{(i,j)\in B_{t}}(X_{ij}^{t}-M_{ij})^{2}\quad\text{s.t.}\quad \lVert X^{t}\rVert_{*}\leq k\quad \forall t\in[T],\quad \sum_{t=1}^{T}\text{Tr}\left(G^{t}X^{t}\right)\leq 0.
\end{align*}
Here, $X^1,\ldots, X^T$ are the sequence of $m\times n$ matrices we choose online. Constraint $\lVert X^{t}\rVert_{*}\leq k$ where $\lVert\cdot\rVert_{*}$ is the nuclear norm induces that each $X^t$ has a low rank. Matrix $G^t$ is randomly generated, and $\sum_{t=1}^{T}\text{Tr}\left(G^{t}X^{t}\right)\leq 0$ is the long-term constraint that we impose. Furthermore, each $B_t$ consists of $b$ entries of $M$.

For our experiments, we test instances with $(m,n,k,b)=(50,50,5,100)$. For each of the instances, the underlying matrix $M$ is randomly chosen to have nuclear norm $1$, and thus, it is contained in the domain $\mathcal{X}=\{X\in\mathbb{R}^{m\times n}:\lVert X\rVert_{*}\leq k\}$. At each iteration $t$, we sample matrix $G^t$
from the uniform distribution over $[-1,1]^{m\times n}$ and sample subset $B_{t}$ from $\{B\subseteq\{1,\ldots,m\times n\}:|B|=b\}$ uniformly at random. Therefore, loss and constraint functions are stochastic and smooth. We test instances with time horizon $T\in\{10,20,\ldots,90,100,200,\ldots,900,1000\}$, and for each value of $T$, we generate 30 instances. 

\begin{figure}
\centering
\begin{minipage}{0.4\textwidth}
\includegraphics[width=\textwidth]{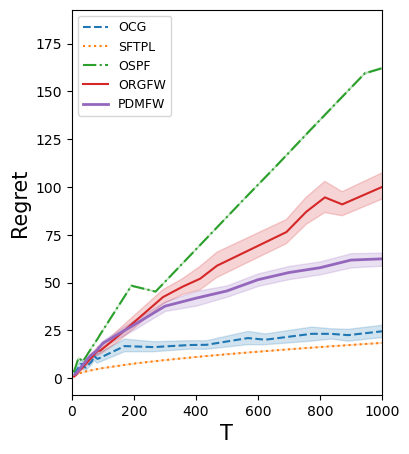}
\subcaption{\label{fig1} }
\end{minipage}
\begin{minipage}{0.4\textwidth}
\includegraphics[width=\textwidth]{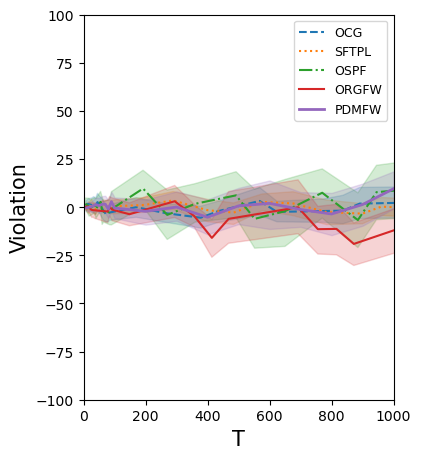}
\subcaption{\label{fig3} }
\end{minipage} 
\caption{\label{figure} Results from testing Algorithms~\ref{alg1} and~\ref{alg2} on  matrix completion instances}
\end{figure}
\Cref{fig1} shows the regret values of the algorithms. PDMFW is \Cref{alg2} while the others correspond to \Cref{alg1} with oracles listed in~\Cref{known-bounds}. 
Note that each algorithm exhibits a sublinear growth in regret, as expected from our theoretical results. In particular, Algorithm~\ref{alg1} with $\mathsf{OCG}$ and Algorithm~\ref{alg1} with $\mathsf{SFTPL}$ achieve lower regret values than the others. 
\Cref{fig3} summarizes the constraint violations of the algorithms. We observe that constraint violation values are centered around 0. In fact, there exists some instances where the constraint violation is below 0. This is possible because the benchmark $\bmo$ is set to a solution with $\bar g(\bmo)\leq 0$. Nevertheless, the results shown in \Cref{fig3} support the theoretical results that Algorithms~\ref{alg1} and~\ref{alg2} attain sublinear constraint violations.

\paragraph{Acknowledgements}
This research is supported, in part, by the KAIST Starting Fund (KAIST-G04220016), the FOUR Brain Korea 21 Program (NRF-5199990113928), the National Research Foundation of Korea (NRF-2022M3J6A1063021).

\bibliographystyle{plainnat} 
\bibliography{mybibfile}

\newpage

\appendix

\section{Performance Analysis of Online Primal-Dual Projection-Free Framework (\cref{alg1})}

Let us define filtration $\{\mathcal{F}_{t}:\ t\geq 0\}$ where $\mathcal{F}_0=\{\emptyset,\Omega\}$ and $\mathcal{F}_{t}=\sigma(\bm{\omega_1},\ldots,\bm{\omega_t})$ being the $\sigma$-algebra generated by the set of random samples $\{\bm{\omega_1},\ldots,\bm{\omega_t}\}$. Note that $\bmc$ is $\mathcal{F}_{t-1}$-measurable for all $t\geq 1$. 

For simplicity, we define vector $\bm{x_k^q}$ and functions $f_k^q, g_k^q, h_k^q,L_k^q$ for $q\in [Q]$ and $k\in [K]$ as
\begin{align*}
\bm{x_k^q} = \bmc,\quad f_k^q(\bmx)  = f_t(\bmx),\quad
g_k^q(\bmx)  = g_t(\bmx),\quad
h_k^q(\bmx) = h_t(\bmx),\quad
L_k^q(\bmx)  = L_t(\bmx,\lambda)
\end{align*}
for any $\bmx\in \mathcal{X}$ and $\lambda\geq 0$
where
$t=k + (q-1)K.$
Then it follows that
\begin{align*}
h_k^q(\bmx) &= f_k^q(\bmx) + \lambda_q g_k^q(\bmx),\\
L_k^q(\bmx,\lambda) &= f_k^q(\bmx) + \lambda_q g_k^q(\bmx)- \frac{\theta}{2K}\lambda^2\\
&=h_k^q(\bmx)- \frac{\theta}{2K}\lambda^2.
\end{align*}
Moreover, note that
\begin{align}\label{alg1:eq1}
\begin{aligned}
\sum_{t=(q-1)K+1}^{qK}\left(L_t(\bmc, \lambda_q)-L_t(\bmx, \lambda_q)\right)&=\sum_{k=1}^{K} \left(L_k^q(\bm{x_k^q}, \lambda_q)-L_k^q(\bmx, \lambda_q)\right)\\
&=\sum_{k=1}^{K} \left(h_k^q(\bm{x_k^q})-h_k^q(\bmx)\right)
\end{aligned}
\end{align}

\begin{lemma}\label{alg1:lemma1}
Let $\bm{x_k^q}$ for $k=1,\ldots, K$ be decisions determined by an $(\alpha, C_0, C_1,C_2)$-oracle. Then
$$
\mathbb{E}\left[\sum_{k=1}^{K} \left(h_k^q(\bm{x_k^q})-h_k^q(\bmx)\right)\mid \mathcal{F}_{(q-1)K}\right]\leq(C_0 +C_1D(1+\lambda_q) +C_2L(1+\lambda_q))T^\frac{\alpha}{3-2\alpha}$$
for any $\bmx\in\mathcal{X}$.
\end{lemma}
\begin{proof}
As we assumed that the loss functions $f_1,\ldots, f_T$ and the constraint functions $g_1,\ldots, g_T$ are $D$-Lipschitz and $L$-smooth, it follows that $h_1,\ldots, h_T$ are $D(1+\lambda_q)$-Lipschitz and $L(1+\lambda_q)$-smooth. Then
\begin{align*}
\mathbb{E}\left[\sum_{k=1}^{K} \left(h_k^q(\bm{x_k^q})-h_k^q(\bmx)\right)\mid \mathcal{F}_{(q-1)K}\right]&\leq \mathbb{E}\left[\sum_{k=1}^{K} \left(h_k^q(\bm{x_k^q})-\min_{\bmx\in\mathcal{X}}h_k^q(\bmx)\right)\mid \mathcal{F}_{(q-1)K}\right]\\
&\leq(C_0 +C_1D(1+\lambda_q) +C_2L(1+\lambda_q))K^\alpha.
\end{align*}
Here, we have $K=T^{\frac{1}{3-2\alpha}}$, as required.
\end{proof}

Next, observe that
\begin{align}\label{alg1:eq2}
\begin{aligned}
\sum_{t=(q-1)K+ 1}^{qK}\left(L_t(\bmc, \lambda)-L_t(\bmc, \lambda_q)\right)&=\sum_{k=1}^{K} \left(L_k^q(\bm{x_k^q}, \lambda)-L_k^q(\bm{x_k^q}, \lambda_q)\right).
\end{aligned}
\end{align}
\begin{lemma}\label{alg1:lemma2}
Let $\lambda_1,\ldots,\lambda_Q$ be the dual variables chosen by \cref{alg1}. Then for any $\lambda\geq 0$,
\begin{align*}
\sum_{q=1}^Q\sum_{k=1}^{K} \left(L_k^q(\bm{x_k^q}, \lambda)-L_k^q(\bm{x_k^q}, \lambda_q)\right)\leq \frac{1}{2\mu}\lambda^2+G^2K^2Q\mu +\theta^2\mu\sum_{q=1}^Q\lambda_q^2.
\end{align*}
\end{lemma}
\begin{proof}
Note that 
$$\lambda_{q+1} = \left[\lambda_q + \mu\sum_{k=1}^{K}\grad_\lambda L_k^q(\bm{x_k^q},\lambda_q)\right]_+.$$
Then
\begin{align*}
\left(\lambda_{q+1}-\lambda\right)^2 &\leq \left(\lambda_q + \mu\sum_{k=1}^{K}\grad_\lambda L_k^q(\bm{x_k^q},\lambda_q)-\lambda\right)^2\\
&=\left(\lambda_{q}-\lambda\right)^2+\mu^2\left(\sum_{k=1}^{K}g_k^q(\bm{x_k^q}) -\theta\lambda_q\right)^2 + 2\mu \left(\sum_{k=1}^{K}\grad_\lambda L_k^q(\bm{x_k^q},\lambda_q)\right)^\top (\lambda_q-\lambda)\\
&\leq \left(\lambda_{q}-\lambda\right)^2 +  2\mu^2(G^2K^2 +\theta^2\lambda_q^2) +2\mu\sum_{k=1}^{K}\left(L_k^q(\bm{x_k^q},\lambda_q)-L_k^q(\bm{x_k^q},\lambda)\right).
\end{align*}
Therefore, it follows that
\begin{align*}
&\sum_{q=1}^Q\sum_{k=1}^{K}\left(L_k^q(\bm{x_k^q},\lambda_q)-L_k^q(\bm{x_k^q},\lambda)\right)\\&\leq \sum_{q=1}^Q\frac{1}{2\mu}\left((\lambda_q-\lambda)^2-(\lambda_{q+1}-\lambda)^2\right) +G^2K^2Q\mu +\theta^2 \mu\sum_{q=1}^Q\lambda_q^2\\
&=\frac{1}{2\mu}\lambda^2  - \frac{1}{2\mu}(\lambda_{Q+1}-\lambda)^2+G^2K^2Q\mu +\theta^2 \mu\sum_{q=1}^Q\lambda_q^2\\
&\leq \frac{1}{2\mu}\lambda^2 +G^2K^2Q\mu +\theta^2 \mu\sum_{q=1}^Q\lambda_q^2
\end{align*}
where the last inequality holds because $(\lambda_{Q+1}-\lambda)^2\geq 0$. 
\end{proof}

Note that
\begin{align}\label{alg1:eq3}
\begin{aligned}
&\sum_{q=1}^Q\sum_{t=(q-1)K+ 1}^{qK}\left(L_t(\bmc, \lambda)-L_t(\bmx, \lambda_q)\right)\\
&=\sum_{q=1}^Q\sum_{k=1}^{K} \left(L_k^q(\bm{x_k^q}, \lambda)-L_k^q(\bmx, \lambda_q)\right)\\
&=\sum_{q=1}^Q\sum_{k=1}^{K} \left(f_k^q(\bm{x_k^q})-f_k^q(\bmx)\right) +\lambda \sum_{q=1}^Q\sum_{k=1}^{K} g_k^q(\bm{x_k^q})-\sum_{q=1}^Q\lambda_q\sum_{k=1}^{K} g_k^q(\bmx) -\frac{Q\theta}{2}\lambda^2  + \frac{\theta}{2}\sum_{q=1}^Q\lambda_q^2.
\end{aligned}
\end{align}
Based on~\eqref{alg1:eq1},~\eqref{alg1:eq2},~\eqref{alg1:eq3}, Lemmas~\ref{alg1:lemma1} and~\ref{alg1:lemma2}, we deduce the following lemma.

\begin{lemma}\label{alg1:lemma3}
Let $\bm{x_k^q}$ for $k=1,\ldots, K$ and $\lambda_1,\ldots,\lambda_Q$ be the decisions and the dual variable chosen by~\Cref{alg1}. Then for any $\bmx\in\mathcal{X}$ and $\lambda\geq 0$, 
\begin{align*}
&\mathbb{E}\left[\sum_{q=1}^Q\sum_{k=1}^{K} \left(f_k^q(\bm{x_k^q})-f_k^q(\bmx)\right) +\lambda \sum_{q=1}^Q\sum_{k=1}^{K} g_k^q(\bm{x_k^q})-\sum_{q=1}^Q\lambda_q\sum_{k=1}^{K} g_k^q(\bmx)\right]\\
&\leq \frac{1}{2}\left(\frac{1}{\mu} + Q\theta\right)\lambda^2 +G^2K^2Q\mu +Q(C_0+C_1D+C_2L)T^{\frac{\alpha}{3-2\alpha}}+\frac{Q}{2}(C_1D+C_2L)T^{\frac{\alpha}{3-2\alpha} +\beta}.
\end{align*}
\end{lemma}
\begin{proof}
Note that
\begin{align*}
&\sum_{q=1}^Q\sum_{k=1}^{K} \left(f_k^q(\bm{x_k^q})-f_k^q(\bmx)\right) +\lambda \sum_{q=1}^Q\sum_{k=1}^{K} g_k^q(\bm{x_k^q})-\sum_{q=1}^Q\lambda_q\sum_{k=1}^{K} g_k^q(\bmx)\\ 
&=\sum_{q=1}^Q\sum_{k=1}^{K} \left(L_k^q(\bm{x_k^q}, \lambda)-L_k^q(\bmx, \lambda_q)\right)+\frac{Q\theta}{2}\lambda^2  -\frac{\theta}{2}\sum_{q=1}^Q\lambda_q^2\\
&=\sum_{q=1}^Q\sum_{k=1}^{K} \left(h_k^q(\bm{x_k^q})-h_k^q(\bmx)\right)+\sum_{q=1}^Q\sum_{k=1}^{K} \left(L_k^q(\bm{x_k^q}, \lambda)-L_k^q(\bm{x_k^q}, \lambda_q)\right)+\frac{Q\theta}{2}\lambda^2  -\frac{\theta}{2}\sum_{q=1}^Q\lambda_q^2\\
&\leq \sum_{q=1}^Q\sum_{k=1}^{K} \left(h_k^q(\bm{x_k^q})-h_k^q(\bmx)\right)+ \frac{1}{2}\left(\frac{1}{\mu} + Q\theta\right)\lambda^2 +G^2K^2Q\mu +\left(\theta^2\mu - \frac{\theta}{2}\right)\sum_{q=1}^Q \lambda_q^2
\end{align*}
where the first equality is from~\eqref{alg1:eq3}, the second equality is due to~\eqref{alg1:eq1} and~\eqref{alg1:eq2}, and the last inequality follows from~\Cref{alg1:lemma2}. By~\Cref{alg1:lemma1},
\begin{align*}
\mathbb{E}\left[\sum_{q=1}^Q\sum_{k=1}^{K} \left(h_k^q(\bm{x_k^q})-h_k^q(\bmx)\right)\right]&=\sum_{q=1}^Q\mathbb{E}\left[\mathbb{E}\left[\sum_{k=1}^{K} \left(h_k^q(\bm{x_k^q})-h_k^q(\bmx)\right)\mid \mathcal{F}_{(q-1)K}\right]\right]\\
&\leq \sum_{q=1}^Q\mathbb{E}\left[(C_0 +C_1D(1+\lambda_q) +C_2L(1+\lambda_q))T^\frac{\alpha}{3-2\alpha}\right].
\end{align*}

Here, using the fact that $a+b\geq 2\sqrt{ab}$ for any $a,b\geq0$, we obtain
$$\lambda_q T^{\frac{\alpha}{3-2\alpha}}\leq \frac{1}{2}T^{\frac{\alpha}{3-2\alpha}+\beta}+\frac{\lambda_q^2}{2}T^{\frac{\alpha}{3-2\alpha}-\beta}.$$
Therefore, it follows that
\begin{align*}
&\sum_{q=1}^Q(C_0 +C_1D(1+\lambda_q) +C_2L(1+\lambda_q))T^\frac{\alpha}{3-2\alpha}\\
&\leq Q(C_0+C_1D+C_2L)T^{\frac{\alpha}{3-2\alpha}}+\frac{Q}{2}(C_1D+C_2L)T^{\frac{\alpha}{3-2\alpha} +\beta}+\frac{1}{2}(C_1D+C_2L)T^{\frac{\alpha}{3-2\alpha}-\beta}\sum_{q=1}^Q\lambda_q^2\\
&=Q(C_0+C_1D+C_2L)T^{\frac{\alpha}{3-2\alpha}}+\frac{Q}{2}(C_1D+C_2L)T^{\frac{\alpha}{3-2\alpha} +\beta}+\frac{\theta}{6}\sum_{q=1}^Q\lambda_q^2.
\end{align*}
Note that
$$\left(\theta^2\mu - \frac{\theta}{2}\right)\sum_{q=1}^Q \lambda_q^2+ \frac{\theta}{6}\sum_{q=1}^Q\lambda_q^2=\left(\theta^2\mu - \frac{\theta}{3}\right)\sum_{q=1}^Q \lambda_q^2=\left(\frac{\theta}{Q+1} - \frac{\theta}{3}\right)\sum_{q=1}^Q \lambda_q^2.$$
Since $\lambda_1=0$ and $Q+1\geq 3$ for $Q\geq 2$, we have
$$\left(\frac{\theta}{Q+1} - \frac{\theta}{3}\right)\sum_{q=1}^Q \lambda_q^2\leq 0,$$
as required.
\end{proof}

\begin{proof}[\bf Proof of \Cref{thm:alg1}]
By plugging in \eqref{alg1-parameters} to the inequality given in~\Cref{alg1:lemma3}, it follows that
\begin{align*}
&\mathbb{E}\left[\sum_{q=1}^Q\sum_{k=1}^{K} \left(f_k^q(\bm{x_k^q})-f_k^q(\bmx)\right) +\lambda \sum_{q=1}^Q\sum_{k=1}^{K} g_k^q(\bm{x_k^q})-\sum_{q=1}^Q\lambda_q\sum_{k=1}^{K} g_k^q(\bmx)\right]\\
&\leq \frac{1}{2}\left(\frac{1}{\mu} + Q\theta\right)\lambda^2 +G^2K^2Q\mu +Q(C_0+C_1D+C_2L)T^{\frac{\alpha}{3-2\alpha}}+\frac{Q}{2}(C_1D+C_2L)T^{\frac{\alpha}{3-2\alpha} +\beta}\\
&=\frac{1}{2}(2Q+1)\theta\lambda^2 +\frac{G^2K^2Q}{\theta(Q+1)} +Q(C_0+C_1D+C_2L)T^{\frac{\alpha}{3-2\alpha}}+\frac{Q}{2}(C_1D+C_2L)T^{\frac{\alpha}{3-2\alpha} +\beta}\\
&\leq 2Q\theta \lambda^2 + \frac{G^2K^2}{\theta}+Q(C_0+C_1D+C_2L)T^{\frac{\alpha}{3-2\alpha}}+\frac{Q}{2}(C_1D+C_2L)T^{\frac{\alpha}{3-2\alpha} +\beta}\\
&=6(C_1D+C_2L)T^{\frac{2-\alpha}{3-2\alpha}-\beta}\lambda^2+\frac{G^2}{3(C_1D+C_2L)}T^{\frac{2-\alpha}{3-2\alpha}+\beta}\\
&\quad + (C_0+C_1D+C_2L)T^{\frac{2-\alpha}{3-2\alpha}}+\frac{1}{2}(C_1D+C_2L)T^{\frac{2-\alpha}{3-2\alpha}+\beta}\\
&\leq 6(C_1D+C_2L)T^{\frac{2-\alpha}{3-2\alpha}-\beta}\lambda^2+\left(\frac{G^2}{3(C_1D+C_2L)}+\frac{3(C_0+C_1D+C_2L)}{2}\right)T^{\frac{2-\alpha}{3-2\alpha}+\beta}.
\end{align*}
Moreover, as $\bar g(\bmo)\leq 0$,
\begin{align*}
\mathbb{E}\left[\sum_{q=1}^Q\lambda_q\sum_{k=1}^{K} g_k^q(\bmo)\right]&=\sum_{q=1}^Q\mathbb{E}\left[\mathbb{E}\left[\lambda_q\sum_{k=1}^{K} g_k^q(\bmo)\mid \mathcal{F}_{(q-1)K}\right]\right]\\
&=\sum_{q=1}^Q\mathbb{E}\left[\lambda_q\sum_{k=1}^{K}\mathbb{E}\left[ g_k^q(\bmo)\mid \mathcal{F}_{(q-1)K}\right]\right]\\
&=\sum_{q=1}^Q\mathbb{E}\left[\lambda_q\sum_{k=1}^{K}\bar g(\bmo)\right]\\
&\leq 0.
\end{align*}
This implies that
\begin{align*}
&\mathbb{E}\left[\sum_{q=1}^Q\sum_{k=1}^{K} \left(f_k^q(\bm{x_k^q})-f_k^q(\bmo)\right) +\lambda \sum_{q=1}^Q\sum_{k=1}^{K} g_k^q(\bm{x_k^q})\right]\\
&\leq 6(C_1D+C_2L)T^{\frac{2-\alpha}{3-2\alpha}-\beta}\lambda^2+\left(\frac{G^2}{3(C_1D+C_2L)}+\frac{3(C_0+C_1D+C_2L)}{2}\right)T^{\frac{2-\alpha}{3-2\alpha}+\beta}.
\end{align*}
Next, we set 
$$\lambda = \frac{1}{12(C_1D+C_2L)T^{\frac{2-\alpha}{3-2\alpha}-\beta}}\left[\mathbb{E}\left[\sum_{q=1}^Q\sum_{k=1}^{K} g_k^q(\bm{x_k^q})\right]\right]_+.$$
Then we obtain
\begin{align*}
\mathbb{E}\left[\sum_{q=1}^Q\sum_{k=1}^{K} \left(f_k^q(\bm{x_k^q})-f_k^q(\bmo)\right) \right]
&\leq -\frac{1}{24(C_1D+C_2L)T^{\frac{2-\alpha}{3-2\alpha}-\beta}}\left[\mathbb{E}\left[\sum_{q=1}^Q\sum_{k=1}^{K} g_k^q(\bm{x_k^q})\right]\right]_+^2\\&\quad +\left(\frac{G^2}{3(C_1D+C_2L)}+\frac{3(C_0+C_1D+C_2L)}{2}\right)T^{\frac{2-\alpha}{3-2\alpha}+\beta}.
\end{align*}
As a result,
\begin{align*}
\mathbb{E}\left[\sum_{q=1}^Q\sum_{k=1}^{K} \left(f_k^q(\bm{x_k^q})-f_k^q(\bmo)\right) \right]
&\leq \left(\frac{G^2}{3(C_1D+C_2L)}+\frac{3(C_0+C_1D+C_2L)}{2}\right)T^{\frac{2-\alpha}{3-2\alpha}+\beta}.
\end{align*}
Moreover,
\begin{align*}
&\frac{1}{24(C_1D+C_2L)T^{\frac{2-\alpha}{3-2\alpha}-\beta}}\left[\mathbb{E}\left[\sum_{q=1}^Q\sum_{k=1}^{K} g_k^q(\bm{x_k^q})\right]\right]_+^2\\
&\leq -\mathbb{E}\left[\sum_{q=1}^Q\sum_{k=1}^{K} \left(f_k^q(\bm{x_k^q})-f_k^q(\bmo)\right) \right]+\left(\frac{G^2}{3(C_1D+C_2L)}+\frac{3(C_0+C_1D+C_2L)}{2}\right)T^{\frac{2-\alpha}{3-2\alpha}+\beta}\\
&\leq TDR+ \left(\frac{G^2}{3(C_1D+C_2L)}+\frac{3(C_0+C_1D+C_2L)}{2}\right)T^{\frac{2-\alpha}{3-2\alpha}+\beta}
\end{align*}
where the second inequality holds because $f_1,\ldots, f_T$ are $D$-Lipschitz. Then
\begin{align*}
&\mathbb{E}\left[\sum_{q=1}^Q\sum_{k=1}^{K} g_k^q(\bm{x_k^q})\right]\\
&\leq\left[\mathbb{E}\left[\sum_{q=1}^Q\sum_{k=1}^{K} g_k^q(\bm{x_k^q})\right]\right]_+\\
&\leq \sqrt{24(C_1D+C_2L)T^{\frac{2-\alpha}{3-2\alpha}-\beta}\left(TDR+ \left(\frac{G^2}{3(C_1D+C_2L)}+\frac{3(C_0+C_1D+C_2L)}{2}\right)T^{\frac{2-\alpha}{3-2\alpha}+\beta}\right)}\\
&=O\left(T^{\frac{5-3\alpha}{6-4\alpha}-\frac{\beta}{2}}\right).
\end{align*}
Then the result follows as
\begin{align*}
\mathbb{E}\left[\sum_{t=1}^T \left(f_t(\bmc)-f_t(\bmo)\right) \right]&=\mathbb{E}\left[\sum_{q=1}^Q\sum_{k=1}^{K} \left(f_k^q(\bm{x_k^q})-f_k^q(\bmo)\right) \right],\\
\mathbb{E}\left[\sum_{t=1}^T g_t(\bmc) \right]&=\mathbb{E}\left[\sum_{q=1}^Q\sum_{k=1}^{K} g_k^q(\bm{x_k^q}) \right],\\
\end{align*}
as required.
\end{proof}

\section{Performance Analysis of PMFWG (\cref{alg2})}

We provide the proof of~\cref{regret:ftpl} in \cref{ftpl:proof}. Then, in \cref{cmfwg:regret-proof} we show \cref{cmfwg:regret}.

\subsection{Regret of FTPL: Proof of \cref{regret:ftpl}}\label{ftpl:proof}

Recall that our adaptive variant of Follow-The-Perturbed-Leader proceeds as the following setup:
\begin{itemize}
\item Choose a random perturbation vector $\bmp$ from $[0,1/\delta]^d$ for some $\delta>0$ uniformly at random.
\item For $t=1,\ldots,T$:
\begin{itemize}
    \item Choose
    \[ \bmc \in \argmin_{\bmx \in \mathcal{X}} \left( \bmp + \sum_{s=1}^{t-1} \bmws \right)^\top \bmx. \]
    \item Observe $\bmwt$.
\end{itemize}
\end{itemize}

Define $M(\bmw) = \argmin_{\bmx \in \mathcal{X}} \bmw^\top \bmx$ for $\bmw\in\mathbb{R}^d$.

\begin{lemma}\label{lemma:BTLregret-bound-1}
The sequence $\{\bmwt\}_{t\in[T]}$ satisfies the following.
\[ \sum_{t=1}^T M\left( \sum_{s=1}^{t} \bmws \right)^\top \bmwt \leq \sum_{t=1}^T M\left( \sum_{s=1}^{T} \bmws\right)^\top \bmwt. \]
\end{lemma}
\begin{proof}
We proceed by induction. The case $T = 1$ is obvious. Now assume it holds for $T \geq 1$, i.e.,
\[ \sum_{t=1}^T M\left( \sum_{s=1}^{t} \bmws \right)^\top \bmwt \leq \sum_{t=1}^T M\left( \sum_{s=1}^{T} \bmws \right)^\top \bmwt. \]
Consider
\begin{align*}
    \sum_{t=1}^{T+1} M\left( \sum_{s=1}^{t} \bmws \right)^\top \bmwt &= \sum_{t=1}^T M\left( \sum_{s=1}^{t} \bmws \right)^\top \bmwt + M\left( \sum_{s=1}^{T+1} \bmws \right)^\top \bm{w_{T+1}}\\
    &\leq M\left( \sum_{s=1}^{T+1} \bmws\right)^\top \bm{w_{T+1}} + \sum_{t=1}^T M\left( \sum_{s=1}^{T} \bmws \right)^\top \bmwt\\
    &= M\left( \sum_{s=1}^{T+1} \bmws \right)^\top \bm{w_{T+1}} + M\left( \sum_{s=1}^{T} \bmws \right)^\top \left( \sum_{t=1}^T \bmwt \right)\\
    &\leq M\left( \sum_{s=1}^{T+1} \bmws \right)^\top \bm{w_{T+1}} + M\left( \sum_{s=1}^{T} \bmws + \bm{w_{T+1}} \right)^\top \left( \sum_{t=1}^T \bmwt \right)\\
    &= \sum_{t=1}^{T+1} M\left( \sum_{s=1}^{T+1} \bmws \right)^\top \bmwt
\end{align*}
where the first inequality comes from the induction hypothesis and the second inequality is because of the definition of $M(\cdot)$, as required.
\end{proof}

\begin{lemma}\label{lemma:BTLregret-bound-2}
Define $\bm{p_0} := \bm{0}$. Then for any sequence $\bm{p_1},\ldots,\bm{p_T}$ we have
\[ \sum_{t=1}^T M\left( \bm{p_t} + \sum_{s=1}^{t} \bmws \right)^\top \bmwt \leq \sum_{t=1}^T M\left( \sum_{s=1}^{T} \bmws \right)^\top \bmwt + R \sum_{t=1}^T \|\bm{p_t} - \bm{p_{t-1}}\|_{\infty}. \]
Consequently, if $\bm{p_t} = \bmp$ for all $t \in [T]$, we have
\[ \sum_{t=1}^T M\left( \bmp + \sum_{s=1}^{t} \bmws \right)^\top \bmwt \leq \sum_{t=1}^T M\left( \sum_{s=1}^{T} \bmws\right)^\top \bmwt + R \|\bmp\|_{\infty}. \]
\end{lemma}
\begin{proof}
We invoke \cref{lemma:BTLregret-bound-1} with $\bm{w_t'} = \bmwt + \bm{p_t} - \bm{p_{t-1}}$ to obtain
\begin{align*}
    &\sum_{t=1}^T M\left( \bm{p_t} + \sum_{s=1}^{t} \bmws \right)^\top (\bmwt + \bm{p_t} - \bm{p_{t-1}})\\
    &\leq \sum_{t=1}^T M\left( \bm{p_T} + \sum_{s=1}^{T} \bmws \right)^\top (\bmwt + \bm{p_t} - \bm{p_{t-1}})\\
    &= M\left( \bm{p_T} + \sum_{s=1}^{T} \bmws \right)^\top \left(\bm{p_T}+ \sum_{t=1}^T \bmwt \right) \\
    &\leq M\left( \sum_{s=1}^{T} \bmws \right)^\top \left(  \bm{p_T}+\sum_{t=1}^T \bmwt\right)\\
    &= M\left( \sum_{s=1}^{T} \bmws \right)^\top \left(  \sum_{t=1}^T \bmwt\right)+ M\left( \sum_{s=1}^{T} \bmws \right)^\top \left( \sum_{t=1}^T (\bm{p_t} - \bm{p_{t-1}}) \right)
\end{align*}
where the second inequality follows from the definition of $M(\cdot)$.
This inequality implies that
\begin{align*}
    &\sum_{t=1}^T M\left( \bm{p_t} + \sum_{s=1}^{t} \bmws \right)^\top \bmwt\\
    &\leq M\left( \sum_{s=1}^{T} \bmws \right)^\top \left(  \sum_{t=1}^T \bmwt\right) + \sum_{t=1}^T \left( M\left( \sum_{s=1}^{T} \bmws \right) - M\left( \bm{p_t} + \sum_{s=1}^{t} \bmws\right) \right)^\top \left(  \bm{p_t} - \bm{p_{t-1}} \right)\\
    &\leq M\left( \sum_{s=1}^{T} \bmws\right)^\top \left(  \sum_{t=1}^T \bmwt \right) + R \sum_{t=1}^T \|\bm{p_t} - \bm{p_{t-1}}\|_{\infty},
\end{align*}
as required.
\end{proof}

\begin{lemma}\label{lemma:FTPL-perturbation-bound}
Let $\bmp$ be a vector sampled from $[0,1/\delta]^n$ uniformly at random, and $\bmw,\bm{w'},\bm{w''}$ be three arbitrary fixed vectors. Then
\[ \bbE \left[ \left( M(\bmw+\bmp) - M(\bmw+\bm{w'}+\bmp) \right)^\top \bm{w''} \right] \leq R \|\bm{w''}\|_\infty \sum_{i=1}^{d} \min\left\{1, \delta|w_i'| \right\}. \]
\end{lemma}
\begin{proof}
We consider the hypercubes $C = \bmw + [0,1/\delta]^d$, $C' = \bmw+\bm{w'}+[0,1/\delta]^d$. Note that $\bm{\tilde{p}} := \bmw+\bmp$ is uniformly distributed over $C$ and $\bm{\tilde{p}'} := \bmw+\bm{w'}+\bmp$ is uniformly distributed over $C'$. Let $C_\cap = C \cap C'$ and $C_\circ = (C \cup C') \setminus (C \cap C')$. Now
\begin{align*}
    &\bbE \left[ \left( M(\bm{\Tilde{p}}) - M(\bm{\Tilde{p}'}) \right)^\top \bm{w''} \right] \\&= \bbE \left[  M(\bm{\tilde{p}}) ^\top \bm{w''} \mid \bm{\tilde{p}} \in C_{\cap} \right] \bbP[ \bm{\tilde{p}} \in C_{\cap}] + \bbE \left[  M(\bm{\tilde{p}})^\top \bm{w''} \mid \bm{\tilde{p}} \in C_{\circ} \right] \bbP[ \bm{\tilde{p}} \in C_{\circ}]\\
    &\quad - \bbE \left[  M(\bm{\tilde{p}'}) ^\top \bm{w''} \mid \bm{\tilde{p}'} \in C_{\cap} \right] \bbP[ \bm{\tilde{p}'} \in C_{\cap}] - \bbE \left[  M(\bm{\tilde{p}'}) ^\top \bm{w''} \mid \bm{\tilde{p}'} \in C_{\circ} \right] \bbP[ \bm{\tilde{p}'} \in C_{\circ}].
\end{align*}
Since $\bmp$ is uniform on the hypercube, the distribution of $\bm{\tilde{p}} \mid \bm{\tilde{p}} \in C_{\cap}$ is equal to the distribution of $\bm{\tilde{p}'} \mid \bm{\tilde{p}'} \in C_{\cap}$, and also the probability that each vector is in $C_{\cap}$ is equal. Therefore,
\begin{align*}
    &\bbE \left[ \left( M(\bm{\tilde{p}}) - M(\bm{\tilde{p}'}) \right)^\top \bm{w''} \right] \\
    &= \bbE \left[  M(\bm{\tilde{p}}) ^\top \bm{w''} \mid \bm{\tilde{p}} \in C_{\circ} \right] \bbP[ \bm{\tilde{p}} \in C_{\circ}] - \bbE \left[  M(\bm{\tilde{p}'}) ^\top \bm{w''} \mid \bm{\tilde{p}'} \in C_{\circ} \right] \bbP[ \bm{\tilde{p}'} \in C_{\circ}]. 
\end{align*}
Since $\bbP[\bm{\tilde{p}} \in C_{\cap}] = \bbP[\bm{\tilde{p}'} \in C_{\cap}]$ we have $\bbP[\bm{\tilde{p}} \in C_{\circ}] = \bbP[\bm{\tilde{p}'} \in C_{\circ}]$, hence
\begin{align*}
    &\bbE \left[ \left( M(\bm{\tilde{p}}) - M(\bm{\tilde{p}'}) \right)^\top \bm{w''} \right] \\&= \bbE \left[  M(\bm{\tilde{p}}) ^\top \bm{w''} \mid \bm{\tilde{p}} \in C_{\circ} \right] \bbP[ \bm{\tilde{p}} \in C_{\circ}] - \bbE \left[  M(\bm{\tilde{p}'}) ^\top \bm{w''} \mid \bm{\tilde{p}} \in C_{\circ} \right] \bbP[ \bm{\tilde{p}'} \in C_{\circ}]\\
    &= \bbP[ \bm{\tilde{p}} \in C_{\circ} ] \left( \bbE \left[  M(\bm{\tilde{p}}) \mid \bm{\tilde{p}} \in C_{\circ} \right] - \bbE \left[  M(\bm{\tilde{p}'}) \mid \bm{\tilde{p}'} \in C_{\circ} \right] \right)^\top \bm{w''}\\
    &\leq R \|\bm{w''}\|_\infty \bbP[ \bm{\tilde{p}} \in C_{\circ} ].
\end{align*}
We now estimate $\bbP[ \bm{\tilde{p}} \in C_{\circ} ]$. If $\bm{\tilde{p}} \in C \setminus C'$ then there exists some $i \in [d]$ such that $\tilde{p}_i \in [w_i,w_i + 1/\delta]$ but $\tilde{p}_i \not\in [w_i + w_i', w_i + w_i' + 1/\delta]$. We can compute that
\[ \bbP[ \tilde{p}_i \in [w_i,w_i + 1/\delta] \setminus [w_i+w_i', w_i+w_i'+1/\delta] ] = \begin{cases}
    1, & |w_i'| \geq 1/\delta \\
    \delta{|w_i'|}, &|w_i'| < 1/\delta
\end{cases} = \min\left\{1, \delta{|w_i'|} \right\}. \]
Now observe that
\begin{align*}
    \bbP\left[ \bm{\tilde{p}} \in C_{\circ} \right] &= \bbP\left[ \bm{\tilde{p}} \in C \setminus C' \right]\\
    &= \bbP\left[ \exists i \in [d] \text{ s.t. } \tilde{p}_i \in [w_i,w_i + 1/\delta] \setminus [w_i+w_i', w_i+w_i'+1/\delta] \right]\\
    &= 1 - \bbP\left[ \forall i \in [d], \  \tilde{p}_i \in [w_i,w_i + 1/\delta] \cap [w_i+w_i', w_i+w_i'+1/\delta] \right]\\
    &= 1 - \prod_{i=1}^d \left( 1 - \bbP[ \tilde{p}_i \in [w_i,w_i + 1/\delta] \setminus [w_i+w_i', w_i+w_i'+1/\delta] ] \right)\\
    &= 1 - \prod_{i=1}^d \left( 1 - \min\left\{1, \delta{|w_i'|} \right\} \right)\\
    &= 1 - \prod_{i=1}^d \max\left\{0, 1- \delta{|w_i'|} \right\}\\
    &= \begin{cases}
        1 - \prod_{i=1}^d \left( 1- \delta{|w_i'|} \right), & \|\bm{w'}\|_{\infty} \leq 1/\delta\\
        1, & \|\bm{w'}\|_{\infty} > 1/\delta,
    \end{cases}
\end{align*}
where the fourth equality follows since each $\Tilde{p}_i$ is independent.
A union bound then gives
\begin{align*}
    \bbP\left[ \bm{\bm{\tilde{p}}} \in C_{\circ} \right] &= \bbP\left[ \bm{\bm{\tilde{p}}} \in C \setminus C' \right]\\
    &= \bbP\left[ \exists i \in [d] \text{ s.t. } \tilde{p}_i \in [w_i,w_i + 1/\delta] \setminus [w_i+w_i', w_i+w_i'+1/\delta] \right]\\
    &\leq \sum_{i=1}^{d} \bbP\left[ \tilde{p}_i \in [w_i,w_i + 1/\delta] \setminus [w_i+w_i', w_i+w_i'+1/\delta] \right]\\
    &= \sum_{i=1}^{d} \min\left\{1, \delta{|w_i'|} \right\},
\end{align*}
as required.
\end{proof}

\begin{theorem}[Follow-the-Perturbed-Leader regret bound]\label{thm:FTPL-regret-bound}
Let $\bm{x_1},\ldots,\bm{x_T}$ be the sequence of decisions chosen by FTPL for the sequence of linear vectors $\bm{w_1},\ldots,\bm{w_T}$. Then
\[ \bbE\left[ \sum_{t=1}^T \bmwt^\top \bmc \right] - \min_{\bmx \in \mathcal{X}} \sum_{t=1}^T \bmwt^\top \bmx \leq \frac{R}{\delta} +  \delta d R \sum_{t=1}^T\|\bmwt\|_\infty^2%
\]
where the expectation is taken with respect to the randomness of the random perturbation vectors.
\end{theorem}
\begin{proof}
According to \cref{lemma:BTLregret-bound-2} we have
\begin{align*}
    \sum_{t=1}^T M\left( \bmp + \sum_{s=1}^{t} \bmws \right)^\top \bmwt &\leq \sum_{t=1}^T M\left( \sum_{s=1}^{T} \bmws \right)^\top \bmwt + R \|\bmp\|_{\infty}\\
    &= \min_{\bmx \in \mathcal{X}} \sum_{t=1}^T \bmwt^\top \bmx + R \|\bmp\|_{\infty}\\
    &\leq \min_{\bmx \in \mathcal{X}} \sum_{t=1}^T \bmwt^\top \bmx + \frac{R}{\delta}
\end{align*}
where the final inequality follows since $\bmp\in[0,1/\delta]^d$. Remember that $\bmc = M\left( \bmp + \sum_{s=1}^{t-1} \bmws \right)$. Therefore,
\begin{align*}
    \sum_{t=1}^T \bmwt^\top \bmc &\leq \min_{\bmx \in \mathcal{X}} \sum_{t=1}^T \bmwt^\top \bmx + \frac{R}{\delta} + \sum_{t=1}^T \left( M\left( \bmp + \sum_{s=1}^{t-1} \bmws \right) - M\left( \bmp + \sum_{s=1}^{t} \bmws \right) \right)^\top \bmwt.
\end{align*}
Take expectations of both sides to get
\begin{align*}
    &\bbE_{\bmp}\left[ \sum_{t=1}^T \bmwt^\top \bmc \right]\\ &\leq \min_{\bmx \in \mathcal{X}} \sum_{t=1}^T \bmwt^\top \bmx + \frac{R}{\delta} + \sum_{t=1}^T \bbE_{\bmp}\left[ M\left( \bmp + \sum_{s=1}^{t-1} \bmws\right) - M\left( \bmp + \sum_{s=1}^{t} \bmws \right) \right]^\top \bmwt.
\end{align*}
We now apply \cref{lemma:FTPL-perturbation-bound} to each term in the last sum to get
\begin{align*}
    \bbE_{\bmp}\left[ \sum_{t=1}^T \bmwt^\top \bmc \right] - \min_{\bmx \in \mathcal{X}} \sum_{t=1}^T \bmwt^\top \bmx &\leq \frac{R}{\delta} + \sum_{t=1}^T R \|\bmwt\|_\infty \sum_{i=1}^{d} \min\left\{1, \delta |w_{t,i}| \right\}\\
    &\leq \frac{R}{\delta} + \sum_{t=1}^T R \|\bmwt\|_\infty \sum_{i=1}^{d} \delta |w_{t,i}| \\
    &\leq \frac{R}{\delta} + \delta dR \sum_{t=1}^T\|\bmwt\|_\infty^2,
\end{align*}
as required.
\end{proof}

Having proved \cref{thm:FTPL-regret-bound}, we can prove \cref{regret:ftpl}. Let us define filtration $\{\mathcal{F}_{t}:\ t\geq 0\}$ where $\mathcal{F}_0=\{\emptyset,\Omega\}$ and $\mathcal{F}_{t}=\sigma(\bm{\omega_1},\ldots,\bm{\omega_t})$ being the $\sigma$-algebra generated by the set of random samples $\{\bm{\omega_1},\ldots,\bm{\omega_t}\}$. Note that $\bmc$ and $\lambda_t$ are $\mathcal{F}_{t-1}$-measurable for all $t\geq 1$.
\begin{proof}[\bf Proof of \cref{regret:ftpl}]
Note that the expected regret of FTPL under the sequence of linear functions $h_1^k(\bm{v}),\ldots,h_T^k(\bm{v})$ can be bounded using \cref{thm:FTPL-regret-bound}, as follows. 
\begin{align}\label{ftpl:eq1}
    \begin{aligned}
        \mathbb{E}\left[\sum_{t=1}^T h_t^k(\bm{v_t^k})-\min_{\bm{v}\in\mathcal{X}}\sum_{t=1}^T h_t^k(\bm{v})\right]&=\mathbb{E}\left[\mathbb{E}\left[\sum_{t=1}^T h_t^k(\bm{v_t^k})\mid \mathcal{F}_T\right]-\min_{\bm{v}\in\mathcal{X}}\sum_{t=1}^T h_t^k(\bm{v})\right]\\
        &\leq \mathbb{E}\left[\frac{R}{\delta} + \delta dR\sum_{t=1}^T \|\grad f_t(\bm{x_t^k})+\lambda_t\grad g_t(\bm{x_t^k})\|_\infty ^2\right]\\
        &\leq \mathbb{E}\left[\frac{R}{\delta} + 2\delta dR\sum_{t=1}^T\left(\|\grad f_t(\bm{x_t^k})\|_\infty^2 + \lambda_t^2\|\grad g_t(\bm{x_t^k})\|_\infty^2\right)\right]
    \end{aligned}
\end{align}
where the first equality holds due to the tower rule, the first inequality follows from~\cref{thm:FTPL-regret-bound}, and the second inequality holds because $\|\bmx + \bmy\|_\infty^2\leq (\|\bmx\|_\infty + \|\bmy\|_\infty)^2\leq 2\|\bmx\|_\infty^2 + 2\|\bmy\|_\infty^2$ for any $\bmx,\bmy\in\mathbb{R}^d$. Note that
\begin{align}\label{ftpl:eq2}
    \begin{aligned}
        &\mathbb{E}\left[\frac{R}{\delta} + 2\delta dR\sum_{t=1}^T\left(\|\grad f_t(\bm{x_t^k})\|_\infty^2 + \lambda_t^2\|\grad g_t(\bm{x_t^k})\|_\infty^2\right)\right]\\
        &=\frac{R}{\delta} +2\delta dR\sum_{t=1}^T\left( \|\grad f_t(\bm{x_t^k})\|_\infty^2 + \mathbb{E}\left[\lambda_t^2\|\grad g_t(\bm{x_t^k})\|_\infty^2\right]\right)\\
        &\leq\frac{R}{\delta} + 2\delta dR\sum_{t=1}^T\left(D^2 + D^2\mathbb{E}\left[\lambda_t^2\right]\right)
    \end{aligned}
\end{align}
where the first inequality follows from~\cref{basic-assumption}. Since $$\delta = \frac{1}{2D\sqrt{d}T^{\frac{1}{2}+\beta}},$$ it follows from~\eqref{ftpl:eq1} that
\begin{align*}
    \mathbb{E}\left[\sum_{t=1}^T h_t^k(\bm{v_t^k})-\min_{\bm{v}\in\mathcal{X}}\sum_{t=1}^T h_t^k(\bm{v})\right]&\leq \frac{R}{\delta}+{2R\delta d}\sum_{t=1}^T(D^2 + D^2\mathbb{E}\left[\lambda_t^2\right])\\
    &=RD\sqrt{d} \left(3T^{\frac{1}{2}+\beta} +   \frac{1}{T^{\frac{1}{2}+\beta}}\mathbb{E}\left[\sum_{t=1}^T\lambda_t^2\right]\right),
\end{align*}
as required.
\end{proof}

\subsection{Regret: Proof of \cref{cmfwg:regret}}\label{cmfwg:regret-proof}

\begin{lemma}\label{product}
Let $\gamma_k=2/(k+1)$ for $k\geq 1$. Then for any $\ell\leq K+1$,
$$
    \prod_{k=\ell}^{K}(1-\gamma_k)\leq\left(\frac{\ell+1}{K+2}\right)^2.$$
\end{lemma}
\begin{proof}
If $\ell= K+1$, then both sides are equal to 1, so the inequality is satisfied. Assume that $\ell\leq K$. Then
$$
\prod_{k=\ell}^{K}(1-\gamma_k)\leq \exp\left(-\sum_{k=\ell}^{K}\frac{2}{k+1}\right)\leq\exp\left(-\int_{x=\ell+1}^{K+2}\frac{2}{x}dx\right)=\left(\frac{\ell+1}{K+2}\right)^2,$$
as required.
\end{proof}

First, we prove the following lemma, which holds because $f_t$ and $g_t$ are smooth (\cref{smoothness}). We closely follow the proof of~\citet[Theorem 1]{meta-frank-wolfe}. We define a constant $C_3$ as $$C_3= 4RD+5R^2L +\frac{R(4D+5RL)^2}{4D\sqrt{d}}+3RD\sqrt{d}.$$
\begin{lemma}\label{lemma2:regret1}
    Let $\gamma_k = 2/(k+1)$. Then for any $\bmx\in\mathcal{X}$,
    \begin{equation*}
        \mathbb{E}\left[\sum_{t=1}^T L_t(\bmc, \lambda_t)- \sum_{t=1}^T L_t(\bmx, \lambda_t)\right]\leq C_3T^{\frac{1}{2}+\beta}+ \frac{2RD\sqrt{d}}{T^{\frac{1}{2}+\beta}}\sum_{t=1}^T\mathbb{E}\left[\lambda_t^2\right].
    \end{equation*}
\end{lemma}
\begin{proof}
    Observe first that
    \begin{equation}\label{eq2:regret1}
        \sum_{t=1}^T L_t(\bmc, \lambda_t)- \sum_{t=1}^T L_t(\bmx, \lambda_t)= \sum_{t=1}^T \left(f_t(\bmc)-f_t(\bmx)\right)+\sum_{t=1}^T\lambda_t\left( g_t(\bmc)-g_t(\bmx)\right).
    \end{equation}
    
    Let $t\geq 1$. Note that for any $1\leq k\leq K$,
    \begin{align}\label{eq2:regret2}
        \begin{aligned}
            &f_t(\bm{x_t^{k+1}}) - f_t(\bmx)\\
            &=f_t(\bm{x_t^k}+\gamma_k(\bm{v_t^k}-\bm{x_t^k})) - f_t(\bmx)\\
            &\leq f_t(\bm{x_t^k}) - f_t(\bmx) + \gamma_k \grad f_t(\bm{x_t^k})^\top (\bm{v_t^k}-\bm{x_t^k}) + \frac{L\gamma_k^2R^2}{2}\\
            &=f_t(\bm{x_t^k}) - f_t(\bmx) + \gamma_k \grad f_t(\bm{x_t^k})^\top (\bmx-\bm{x_t^k}) + \gamma_k \grad f_t(\bm{x_t^k})^\top (\bm{v_t^k}-\bmx)+\frac{L\gamma_k^2R^2}{2}\\
            &\leq (1-\gamma_k)\left(f_t(\bm{x_t^k}) - f_t(\bmx)\right) + \gamma_k \grad f_t(\bm{x_t^k})^\top (\bm{v_t^k}-\bmx)+\frac{L\gamma_k^2R^2}{2}
        \end{aligned}
    \end{align}
    where the first inequality holds because $f_t$ is $L$-smooth and $\|\bm{v_t^k}-\bm{x_t^k}\|_1\leq R$ whereas the second inequality follows from the convexity of $f_t$. Then it follows from~\eqref{eq2:regret2} that
    \begin{align}\label{eq2:regret3}
        \begin{aligned}
            f_t(\bmc) - f_t(\bmx)
            &=f_t(\bm{x_t^{K+1}}) - f_t(\bmx)\\
            &\leq \prod_{k=1}^{K}(1-\gamma_k)\left(f_t(\bm{x_t^1}) - f_t(\bmx)\right) \\
            &\quad + \sum_{k=1}^{K}\gamma_k\left( \grad f_t(\bm{x_t^k})^\top (\bm{v_t^k}-\bmx)+\frac{L\gamma_kR^2}{2}\right)\prod_{j=k+1}^{K}(1-\gamma_j)\\
            &=\prod_{k=1}^{K}(1-\gamma_k)\left(f_t(\bm{x_{t-1}}) - f_t(\bmx)\right) \\
            &\quad + \sum_{k=1}^{K}\gamma_k\left( \grad f_t(\bm{x_t^k})^\top (\bm{v_t^k}-\bmx)+\frac{L\gamma_kR^2}{2}\right)\prod_{j=k+1}^{K}(1-\gamma_j)\\
            &\leq \prod_{k=1}^{K}(1-\gamma_k)DR \\
            &\quad + \sum_{k=1}^{K}\gamma_k\left( \grad f_t(\bm{x_t^k})^\top (\bm{v_t^k}-\bmx)+\frac{L\gamma_kR^2}{2}\right)\prod_{j=k+1}^{K}(1-\gamma_j)
        \end{aligned}
    \end{align}
    where the last inequality holds because $f_t$ is $D$-Lipschitz.
    Similarly, we deduce that
    \begin{align}\label{eq2:regret4}
        \begin{aligned}
            &\mathbb{E}\left[\lambda_t\left(g_t(\bmc) - g_t(\bmx)\right)\right]\\
            &\leq\prod_{k=1}^{K}(1-\gamma_k)DR\mathbb{E}\left[\lambda_t\right] \\
            &\quad + \sum_{k=1}^{K}\gamma_k\left(\mathbb{E}\left[ \lambda_t\grad g_t(\bm{x_t^k})^\top (\bm{v_t^k}-\bmx)\right]+\frac{L\gamma_kR^2}{2}\mathbb{E}\left[\lambda_t\right]\right)\prod_{j=k+1}^{K}(1-\gamma_j).
        \end{aligned}
    \end{align}
    Summing up~\eqref{eq2:regret3} and~\eqref{eq2:regret4} for $t=1,\ldots, T$, we obtain 
    \begin{align}\label{eq2:regret5}
        \begin{aligned}
            &\mathbb{E}\left[\sum_{t=1}^T L_t(\bmc, \lambda_t)- \sum_{t=1}^T L_t(\bmx, \lambda_t)\right]\\
            &\leq \sum_{t=1}^T\prod_{k=1}^{K}(1-\gamma_k)DR\left(1+ \mathbb{E}\left[\lambda_t\right]\right)\\
            &\quad + \sum_{t=1}^T\sum_{k=1}^{K}\gamma_k \mathbb{E}\left[\left(\grad f_t(\bm{x_t^k})+\lambda_t\grad g_t(\bm{x_t^k})\right)^\top (\bm{v_t^k}-\bmx)\right]\prod_{j=k+1}^{K}(1-\gamma_j)\\
            &\quad + \sum_{t=1}^T\sum_{k=1}^{K}\frac{L\gamma_k^2R^2}{2}\left(1+\mathbb{E}\left[\lambda_t\right]\right)\prod_{j=k+1}^{K}(1-\gamma_j).
        \end{aligned}
    \end{align}
    Here, 
        \begin{align}\label{eq2:regret6}
            \begin{aligned}
                & \sum_{t=1}^T\prod_{k=1}^{K}(1-\gamma_k)DR\left(1+ \mathbb{E}\left[\lambda_t\right]\right) +\sum_{t=1}^T\sum_{k=1}^{K}\frac{L\gamma_k^2R^2}{2}\left(1+\mathbb{E}\left[\lambda_t\right]\right)\prod_{j=k+1}^{K}(1-\gamma_j)\\
                &\leq \sum_{t=1}^T \frac{4RD}{(K+2)^2}\left(1+\mathbb{E}\left[\lambda_t\right]\right)+ \sum_{t=1}^T \sum_{k=1}^{K}\frac{2R^2L}{(k+1)^2}\left(\frac{k+2}{K+2}\right)^2\left(1+\mathbb{E}\left[\lambda_t\right]\right)\\
                &\leq \sum_{t=1}^T \frac{4RD}{K+2}\left(1+\mathbb{E}\left[\lambda_t\right]\right) + \sum_{t=1}^T \frac{9R^2LK}{2\left(K+2\right)^2}\left(1+\mathbb{E}\left[\lambda_t\right]\right)\\
                &\leq \sum_{t=1}^T \frac{4RD}{T^{\frac{1}{2}+\beta}}\left(1+\mathbb{E}\left[\lambda_t\right]\right) + \sum_{t=1}^T \frac{5R^2L}{T^{\frac{1}{2}+\beta}}\left(1+\mathbb{E}\left[\lambda_t\right]\right)
        \end{aligned}
        \end{align}
    where the first inequality follows from~\cref{product}, the second inequality holds because $K+2 \geq 1$ and $2(k+2)\leq 3(k+1)$, and the third inequality is obtained using $K+2 \geq T^{\frac{1}{2}+\beta}$ and $K\leq T^{\frac{1}{2}+\beta}$. 	Furthermore,
    \begin{align}\label{eq2:regret7}
        \begin{aligned}
            &\sum_{t=1}^T\sum_{k=1}^{K}\gamma_k\prod_{j=k+1}^{K}(1-\gamma_j)\mathbb{E}\left[\left(\grad f_t(\bm{x_t^k})+\lambda_t\grad g_t(\bm{x_t^k})\right)^\top (\bm{v_t^k}-\bmx)\right]\\
            &=\sum_{k=1}^{K}\gamma_k\prod_{j=k+1}^{K}(1-\gamma_j)\sum_{t=1}^T\mathbb{E}\left[\left(\grad f_t(\bm{x_t^k})+\lambda_t\grad g_t(\bm{x_t^k})\right)^\top (\bm{v_t^k}-\bmx)\right]\\
            &=\sum_{k=1}^{K}\gamma_k\prod_{j=k+1}^{K}(1-\gamma_j)\mathbb{E}\left[\sum_{t=1}^Th_t^k(\bm{v_t^k})-\sum_{t=1}^Th_t^k(\bmx)\right]\\
              &\leq\sum_{k=1}^{K}\gamma_k\prod_{j=k+1}^{K}(1-\gamma_j)\mathcal{R}^{\mathcal{E}}(T)\\
            &\leq \sum_{k=1}^{K}\frac{2}{(k+1)}\left(\frac{k+2}{K+2}\right)^2\mathcal{R}^{\mathcal{E}}(T)\\
             &\leq \sum_{k=1}^{K}\frac{3(k+2)}{(K+2)^2}\mathcal{R}^{\mathcal{E}}(T)\\
             &\leq \frac{3}{2}\mathcal{R}^{\mathcal{E}}(T)
        \end{aligned}
    \end{align}
    where the second inequality is from \cref{product}. By~\cref{regret:ftpl},~\eqref{eq2:regret5},~\eqref{eq2:regret6}, and~\eqref{eq2:regret7},
    \begin{align}\label{eq2:regret8}
        \begin{aligned}
            &\mathbb{E}\left[\sum_{t=1}^T L_t(\bmc, \lambda_t)- \sum_{t=1}^T L_t(\bmx, \lambda_t)\right]\\
            &\leq \sum_{t=1}^T \frac{4RD+5R^2L}{T^{\frac{1}{2}+\beta}}\left(1+\mathbb{E}\left[\lambda_t\right]\right) + RD\sqrt{d} \left(3T^{\frac{1}{2}+\beta} +  \frac{1}{T^{\frac{1}{2}+\beta}}\sum_{t=1}^T\mathbb{E}[\lambda_t^2]\right).
        \end{aligned}
    \end{align}
    Observe that
    \begin{align}\label{eq2:regret9} 
\frac{4RD+5R^2L}{T^{{\frac{1}{2}+\beta}}}\lambda_t \leq \frac{RD\sqrt{d}}{T^{\frac{1}{2}+\beta}}\lambda_t^2 +\frac{(4RD+5R^2L)^2}{4RD\sqrt{d}T^{{\frac{1}{2}+\beta}}}=\frac{RD\sqrt{d}}{T^{\frac{1}{2}+\beta}}\lambda_t^2 +\frac{R(4D+5RL)^2}{4D\sqrt{d}T^{{\frac{1}{2}+\beta}}}.
    \end{align}
as $2\sqrt{pq}\leq p+q$ for any $p,q\geq 0$. 
    Based on~\eqref{eq2:regret8} and~\eqref{eq2:regret9},
    \begin{align*}
            \mathbb{E}\left[\sum_{t=1}^T L_t(\bmc, \lambda_t)- \sum_{t=1}^T L_t(\bmx, \lambda_t)\right]
            &\leq \sum_{t=1}^T\left( 4RD+5R^2L +\frac{R(4D+5RL)^2}{4D\sqrt{d}}\right)\frac{1}{T^{\frac{1}{2}+\beta}}\\
            &\quad +3RD\sqrt{d} T^{\frac{1}{2}+\beta} +\frac{2RD\sqrt{d}}{T^{\frac{1}{2}+\beta}}\sum_{t=1}^T\mathbb{E}\left[\lambda_t^2\right]\\
            &\leq C_3T^{\frac{1}{2}+\beta} + \frac{2RD\sqrt{d}}{T^{\frac{1}{2}+\beta}}\sum_{t=1}^T\mathbb{E}\left[\lambda_t^2\right],
    \end{align*}
as required.
\end{proof}

Next, based on the concavity of $L_t(\bmc,\lambda)$ with respect to $\lambda$, we show the following lemma. 

\begin{lemma}\label{lemma2:regret2}
    
    For any $\lambda\geq0$,
    \begin{equation*}
        \mathbb{E}\left[\sum_{t=1}^T L_t(\bmc, \lambda)- \sum_{t=1}^T L_t(\bmc, \lambda_t)\right]\leq  \frac{1}{2\mu}\lambda^2+G^2T\mu +\mu\theta^2\sum_{t=1}^T\mathbb{E}\left[\lambda_t^2\right].
    \end{equation*}
\end{lemma}
\begin{proof}
    Note that
    \begin{align}\label{eq2-2:1} 
        \begin{aligned}
            &(\lambda_{t+1}-\lambda)^2 \\
            &\leq (\lambda_t+\mu \grad_\lambda L_t(\bmc,\lambda_t)-\lambda)^2\\
            &=(\lambda_t-\lambda)^2 +\mu^2(g_t(\bmc) - \theta\lambda_t)^2 + 2\mu \grad_\lambda L_t(\bmc,\lambda_t)^\top(\lambda_{t}-\lambda)\\
            &\leq(\lambda_t-\lambda)^2 +\mu^2(g_t(\bmc) - \theta\lambda_t)^2 - 2\mu \left(L_t(\bmc, \lambda) - L_t(\bmc, \lambda_t)\right)- {\mu\theta}(\lambda_t-\lambda)^2
        \end{aligned}
    \end{align}
    where the first inequality is because $\lambda_{t+1}$ is the projection of $\lambda_t+\mu \grad_\lambda L_t(\bmc,\lambda_t)$, the equality is because $\grad_\lambda L_t(\bmc,\lambda_t)=g_t(\bmc) - \theta\lambda_t$, the second inequality holds because $L_t$ is $\theta$-strongly concave with respect to $\lambda$.
    Note that
    \begin{align}\label{eq2-2:2} 
        \begin{aligned}
            \mathbb{E}\left[(g_t(\bmc) - \theta\lambda_t)^2\right]&\leq \mathbb{E}\left[2(g_t(\bmc))^2 +2 \theta^2\lambda_t^2\right]\\
            &=2\mathbb{E}\left[\mathbb{E}\left[(g_t(\bmc))^2 +2 \theta^2\lambda_t^2\mid\mathcal{F}_{t-1}\right]\right]
            \\&
            \leq 2G^2+2 \theta^2\mathbb{E}\left[\lambda_t^2\right].
        \end{aligned}
    \end{align}
    Combining~\eqref{eq2-2:1} and~\eqref{eq2-2:2},
    \begin{align}\label{eq2-2:3} 
        \begin{aligned}
            &\mathbb{E}\left[(\lambda_{t+1}-\lambda)^2 \right]\\
            &\leq\bbE\left[(\lambda_t-\lambda)^2\right] +\mu^2(2G^2+2\theta^2\lambda_t^2) - \mathbb{E}\left[2\mu \left(L_t(\bmc, \lambda) - L_t(\bmc, \lambda_t)\right)\right]-{\mu\theta}(\lambda_t-\lambda)^2.
        \end{aligned}
    \end{align}
    Then it follows from~\eqref{eq2-2:3} that
    \begin{align*}
        \begin{aligned}
            &\mathbb{E}\left[\sum_{t=1}^T L_t(\bmc, \lambda)- \sum_{t=1}^T L_t(\bmc, \lambda_t)\right]\\
            &\leq\mathbb{E}\left[\sum_{t=1}^T\frac{1}{2\mu}\left((\lambda_t-\lambda)^2-(\lambda_{t+1}-\lambda)^2\right)+\sum_{t=1}^T{\mu}(G^2+\theta^2\lambda_t^2)-\sum_{t=1}^T \frac{\theta}{2}(\lambda_t-\lambda)^2\right]\\
            &\leq\frac{1}{2\mu}(\lambda_1-\lambda)^2+
            G^2T\mu +\sum_{t=1}^T\mu\theta^2\mathbb{E}\left[\lambda_t^2\right]\\
            &=\frac{1}{2\mu}\lambda^2+G^2T\mu +\sum_{t=1}^T\mu\theta^2\mathbb{E}\left[\lambda_t^2\right]
        \end{aligned}
    \end{align*}
    where the last inequality holds due to $\lambda_1=0$. Therefore, the assertion of the lemma holds, as required.
\end{proof}

Based on Lemmas~\ref{lemma2:regret1} and~\ref{lemma2:regret2}, we show the following lemma. 
Recall that 
$$C_1= 4RD+5R^2L +\frac{R(4D+5RL)^2}{4D\sqrt{d}}+3RD\sqrt{d}+\frac{G^2}{12RD\sqrt{d}}= C_3 +\frac{G^2}{12RD\sqrt{d}}.$$

\begin{lemma}\label{lemma2:regret3}
    Let $\gamma_k=2/(k+1)$. Then for any $\bmx\in\mathcal{X}$ and $\lambda\geq 0$,
    \begin{align*}
        &\mathbb{E}\left[\left(\sum_{t=1}^T f_t(\bmc) - \sum_{t=1}^T f_t(\bmx)\right) +\left(\sum_{t=1}^T \lambda g_t(\bmc)-\sum_{t=1}^T\lambda_tg_t(\bmx)\right)\right]\\
        &\leq C_1T^{\frac{1}{2}+\beta} + 24RD\sqrt{d}T^{\frac{1}{2}-\beta}\lambda^2.
    \end{align*}

\end{lemma}
\begin{proof}
    Adding the two inequalities proved in Lemmas~\ref{lemma2:regret1} and~\ref{lemma2:regret2}, we obtain
    \begin{align}\label{eq2-3:1}
        \begin{aligned}
            &\mathbb{E}\left[\sum_{t=1}^T L_t(\bmc, \lambda)- \sum_{t=1}^T L_t(\bmx, \lambda_t)\right]\\
            &\leq C_3T^{\frac{1}{2}+\beta} + \frac{2RD\sqrt{d}}{T^{\frac{1}{2}+\beta}}\sum_{t=1}^T\mathbb{E}\left[\lambda_t^2\right] +\frac{1}{2\mu}\lambda^2+G^2T\mu +\sum_{t=1}^T\mu\theta^2\mathbb{E}\left[\lambda_t^2\right].
        \end{aligned}
    \end{align}
    By~\eqref{eq2-3:1},
    \begin{align*}
        &\mathbb{E}\left[\left(\sum_{t=1}^T f_t(\bmc) - \sum_{t=1}^T f_t(\bmx)\right) +\left(\sum_{t=1}^T \lambda g_t(\bmc)-\sum_{t=1}^T\lambda_tg_t(\bmx)\right)\right]\\
        &\leq C_3T^{\frac{1}{2}+\beta} + G^2T \mu + \left(\frac{1}{2}T\theta+\frac{1}{2\mu}\right)\lambda^2+\sum_{t=1}^T\left(\frac{2RD\sqrt{d}}{T^\beta}+\mu\theta^2 - \frac{\theta}{2}\right)\mathbb{E}\left[\lambda_t^2\right]\\
        &\leq C_3T^{\frac{1}{2}+\beta}+ G^2T \mu + \left(\frac{1}{2}T\theta+\frac{1}{2\mu}\right)\lambda^2\\
&\leq C_3T^{\frac{1}{2}+\beta} + \frac{G^2}{12RD\sqrt{d}}T^{\frac{1}{2}+\beta} + 24RD\sqrt{d}T^{\frac{1}{2}-\beta}\lambda^2\\
&=C_1T^{\frac{1}{2}+\beta} + 24RD\sqrt{d}T^{\frac{1}{2}-\beta}\lambda^2
    \end{align*}
    where the second inequality holds because $\lambda_1=0$ and for $T\geq1$, 
$$\frac{2RD\sqrt{d}}{T^{\frac{1}{2}+\beta}}+\mu\theta^2 - \frac{\theta}{2}=\frac{\theta}{6} + \frac{\theta}{T+2}-\frac{\theta}{2}\leq \frac{\theta}{6} + \frac{\theta}{3}-\frac{\theta}{2}=0,$$
as required.
\end{proof}

Lastly, we need the following lemma.

\begin{lemma}\label{expectation1}
    For any $\{\lambda_t\}_{t=1}^T\subseteq \mathbb{R}_+$ such that $\lambda_t$ is $\mathcal{F}_{t-1}$-measurable for $t\geq 1$, then
    $$\mathbb{E}\left[ \sum_{t=1}^T\lambda_tg_t(\bmo)\right]\leq 0$$
\end{lemma}
\begin{proof}
    Note that
    \begin{align*}
        \mathbb{E}[\lambda_tg_t(\bmo)]&=\mathbb{E}\left[\mathbb{E}[\lambda_tg_t(\bmo)\mid\mathcal{F}_{t-1}]\right]\\
        &=\mathbb{E}\left[\lambda_t\mathbb{E}[g_t(\bmo)\mid\mathcal{F}_{t-1}]\right]\\
        &=\mathbb{E}\left[\lambda_t\bar g(\bmo)\right]\\
        &=\mathbb{E}\left[\lambda_t\right]\bar g(\bmo) \\
        &\leq 0
    \end{align*}
    where the first equality comes from the tower rule, the second equality holds because $\lambda_t$ is $\mathcal{F}_{t-1}$-measurable, the third equality follows as $g_t$ is independent of $\{\bm{\omega_1},\ldots,\bm{\omega_{t-1}}\}$, the fourth equality holds because $\bar g(\bmo)$ is a constant, and the last inequality follows from $\lambda_t\geq 0$ and $\bar g(\bmo)\leq 0$. Then
    $$\mathbb{E}\left[ \sum_{t=1}^T\lambda_tg_t(\bmo)\right]=\sum_{t=1}^T\mathbb{E}[\lambda_tg_t(\bmo)]\leq 0,$$
    as required.
\end{proof}

Now we are ready to prove \cref{cmfwg:regret}. 

\begin{proof}[\bf Proof of \cref{cmfwg:regret}]
 By \cref{lemma2:regret3},
    \begin{align}\label{eq2-4:1}
        \begin{aligned}
            &\mathbb{E}\left[\left(\sum_{t=1}^T f_t(\bmc) - \sum_{t=1}^T f_t(\bmx)\right) +\left(\sum_{t=1}^T \lambda g_t(\bmc)-\sum_{t=1}^T\lambda_tg_t(\bmx)\right)\right]\\
            &\leq C_1T^{\frac{1}{2}+\beta} + 24RD\sqrt{d}T^{\frac{1}{2}-\beta}\lambda^2.
        \end{aligned}
    \end{align}
    Here, we set
    $$\lambda = \left(48RD\sqrt{d}T^{\frac{1}{2}-\beta}\right)^{-1}\left[\mathbb{E}\left[\sum_{t=1}^T g_t(\bmc)\right]\right]_+$$
    where $(p)_+=\max\{p,0\}$. Then
    $$\mathbb{E}\left[\sum_{t=1}^T \lambda g_t(\bmc)\right]=\left(48RD\sqrt{d}T^{{\frac{1}{2}-\beta}}\right)^{-1}\left[\mathbb{E}\left[\sum_{t=1}^T g_t(\bmc)\right]\right]_+^2=48RD\sqrt{d}T^{{\frac{1}{2}-\beta}}\lambda^2.$$
    Together with~\eqref{eq2-4:1}, this implies
    \begin{align}\label{eq2-4:2}
\begin{aligned}
    &	\mathbb{E}\left[\sum_{t=1}^T f_t(\bmc) - \sum_{t=1}^T f_t(\bmx)\right]\\
&\leq C_1T^{\frac{1}{2}+\beta}  + \mathbb{E}\left[\sum_{t=1}^T\lambda_tg_t(\bmx)\right]-\frac{1}{96RD\sqrt{d}T^{{\frac{1}{2}-\beta}}} \left[\mathbb{E}\left[\sum_{t=1}^T g_t(\bmc)\right]\right]_+^2.
\end{aligned}
    \end{align}
    In particular, we consider $\bmx = \bmo$. Note that $\lambda_t$ is $\mathcal{F}_{t-1}$-measurable for all $t\geq 1$. Then by Lemma~\ref{expectation1},
    \begin{equation}\label{eq2-4:3}
        \mathbb{E}\left[\sum_{t=1}^T\lambda_tg_t(\bmo)\right]\leq 0
    \end{equation}Then it follows from~\eqref{eq2-4:2} and~\eqref{eq2-4:3} that 
    $$\mathbb{E}\left[\sum_{t=1}^Tf_t(\bmc) - \sum_{t=1}^T f_t(\bmo)\right]\leq C_1T^{\frac{1}{2}+\beta} .$$
    as required. Moreover,
Since $f_1,\ldots, f_T$ are $D$-Lipschitz, it follows from~\eqref{eq2-4:2} that
\begin{align*}
\left[\mathbb{E}\left[\sum_{t=1}^T g_t(\bmc)\right]\right]_+^2\leq 96RD\sqrt{d}T^{{\frac{1}{2}-\beta}}\left(RDT + C_1T^{\frac{1}{2}+\beta} \right).
\end{align*}
Therefore,
\begin{align*}
\mathbb{E}\left[\sum_{t=1}^T g_t(\bmc)\right]\leq C_2 T^{{\frac{3}{4}-\frac{\beta}{2}}}
\end{align*}
where $$C_2=\sqrt{96RD\sqrt{d}\left(RD+C_1\right)},$$
as required.
\end{proof}

\section{Analysis for Stochastic Loss Functions under Strong Duality (\cref{sec:stochastic-loss})}

We first prove Lemmas~\ref{expectation0} and~\ref{expectation2}. Then based on these lemmas, we prove Theorems~\ref{thm:Slater1} and~\ref{thm:Slater2}.

\begin{proof}[\bf Proof of \cref{expectation0}]
Note that
\begin{align*}
    \mathbb{E}\left[\sum_{t=1}^T g_t(\bmc) \right]&=\mathbb{E}\left[\sum_{q=1}^Q\sum_{k=1}^K g_{(q-1)K+k}(\bm{x_{(q-1)K+k}}) \right]\\
    &=\sum_{q=1}^Q\sum_{k=1}^K\mathbb{E}\left[ \mathbb{E}\left[g_{(q-1)K+k}(\bm{x_{(q-1)K+k}}) \mid \mathcal{F}_{(q-1)K+k-1}\right] \right]\\
    &=\mathbb{E}\left[\sum_{q=1}^Q\sum_{k=1}^K\bar g(\bm{x_{(q-1)K+k}})\right]\\
    &=\mathbb{E}\left[\sum_{q=1}^Q\sum_{k=1}^K \bar L(\bm{x_{(q-1)K+k}}, \lambda^*+1) - \sum_{q=1}^Q\sum_{k=1}^K \bar L(\bm{x_{(q-1)K+k}}, \lambda^*)\right]\\
    &\leq\mathbb{E}\left[ \sum_{q=1}^Q\sum_{k=1}^K\bar L(\bm{x_{(q-1)K+k}}, \lambda^*+1) - \sum_{q=1}^Q\sum_{k=1}^K \bar L(\bmo, \lambda_{q})\right]
\end{align*}
where the last inequality follows from~\eqref{saddle-point}.
Moreover, note that
\begin{align*}
&\mathbb{E}\left[\sum_{q=1}^Q\sum_{k=1}^K\left(f_{(q-1)K+k}(\bm{x_{(q-1)K+k}})+(\lambda^*+1)g_{(q-1)K+k}(\bm{x_{(q-1)K+k}})\right)\right]\\
&=\sum_{q=1}^Q\sum_{k=1}^K\mathbb{E}\left[\mathbb{E}\left[f_{(q-1)K+k}(\bm{x_{(q-1)K+k}})+(\lambda^*+1)g_{(q-1)K+k}(\bm{x_{(q-1)K+k}})\mid \mathcal{F}_{(q-1)K+k-1}\right]\right]\\
&=\mathbb{E}\left[\sum_{q=1}^Q\sum_{k=1}^K\left(\bar f(\bm{x_{(q-1)K+k}})+(\lambda^*+1)\bar g(\bm{x_{(q-1)K+k}})\right)\right]\\
&=\mathbb{E}\left[ \sum_{q=1}^Q\sum_{k=1}^K\bar L(\bm{x_{(q-1)K+k}}, \lambda^*+1) \right].
\end{align*}
Similarly, 
\begin{align*}
\mathbb{E}\left[\sum_{q=1}^Q\sum_{k=1}^K\left(f_{(q-1)K+k}(\bmo)+\lambda_qg_{(q-1)K+k}(\bmo)\right)\right]
&=\mathbb{E}\left[ \sum_{t=1}^T \bar L(\bmo, \lambda_q) \right].
\end{align*}
Therefore, the assertion follows, as required.
\end{proof}

\begin{proof}[\bf Proof of \cref{expectation2}]
Note that
\begin{align*}
    \mathbb{E}\left[\sum_{t=1}^T g_t(\bmc) \right]&=\sum_{t=1}^T\mathbb{E}\left[ \mathbb{E}\left[g_t(\bmc)\mid \mathcal{F}_{t-1}\right] \right]\\
    &=\mathbb{E}\left[\sum_{t=1}^T\bar g(\bmc)\right]\\
    &=\mathbb{E}\left[\sum_{t=1}^T \bar L(\bmc, \lambda^*+1) - \sum_{t=1}^T \bar L(\bmc, \lambda^*)\right]\\
    &\leq\mathbb{E}\left[ \sum_{t=1}^T \bar L(\bmc, \lambda^*+1) - \sum_{t=1}^T \bar L(\bmo, \lambda_t)\right]
\end{align*}
where the last inequality follows from~\eqref{saddle-point}.
Moreover, note that
\begin{align*}
\mathbb{E}\left[\sum_{t=1}^T\left(f_t(\bmc)+(\lambda^*+1)g_t(\bmc)\right)\right]&=\sum_{t=1}^T\mathbb{E}\left[\mathbb{E}\left[f_t(\bmc)+(\lambda^*+1)g_t(\bmc)\mid \mathcal{F}_{t-1}\right]\right]\\
&=\mathbb{E}\left[\sum_{t=1}^T\left(\bar f(\bmc)+(\lambda^*+1)\bar g(\bmc)\right)\right]\\
&=\mathbb{E}\left[ \sum_{t=1}^T \bar L(\bmc, \lambda^*+1) \right].
\end{align*}
Similarly, 
\begin{align*}
\mathbb{E}\left[\sum_{t=1}^T\left(f_t(\bmo)+\lambda_tg_t(\bmo)\right)\right]
&=\mathbb{E}\left[ \sum_{t=1}^T \bar L(\bmo, \lambda_t) \right].
\end{align*}
Therefore, the assertion follows, as required.
\end{proof}

As we have shown Lemmas~\ref{expectation0} and~\ref{expectation2}, we are ready to prove Theorems~\ref{thm:Slater1} and~\ref{thm:Slater2}.

\begin{proof}[\bf Proof of \cref{thm:Slater1}]
By~\Cref{thm:alg1}, we know that 
$$\mathbb{E}\left[\sum_{t=1}^T f_t(\bmc) -\sum_{t=1}^T f_t(\bmo)\right]=O\left(T^{\frac{2-\alpha}{3-2\alpha}}\right).$$
Therefore, it suffices to prove the upper bound on the constraint violation.
Note that
\begin{align*}
&\mathbb{E}\left[\sum_{t=1}^T g_t(\bmc) \right]\\
&\leq \mathbb{E}\left[\sum_{q=1}^Q\sum_{k=1}^K\left(f_{(q-1)K+k}(\bm{x_{(q-1)K+k}})+(\lambda^*+1)g_{(q-1)K+k}(\bm{x_{(q-1)K+k}})\right)\right]\\
&\quad -\mathbb{E}\left[ \sum_{q=1}^Q\sum_{k=1}^K\left(f_{(q-1)K+k}(\bmo)+\lambda_qg_{(q-1)K+k}(\bmo)\right)\right]\\
&\leq \frac{1}{2}\left(\frac{1}{\mu} + Q\theta\right)\lambda^{*2} +G^2K^2Q\mu +Q(C_0+C_1D+C_2L)T^{\frac{\alpha}{3-2\alpha}}+\frac{Q}{2}(C_1D+C_2L)T^{\frac{\alpha}{3-2\alpha}}\\
&\leq O\left(T^{\frac{2-\alpha}{3-2\alpha}}\right)
\end{align*}
where the first inequality is due to~\Cref{expectation0}, the second inequality is from~\Cref{alg1:lemma3}, and the last inequality holds because of~\eqref{alg1-parameters}.
\end{proof}

\begin{proof}[\bf Proof of \cref{thm:Slater2}]
By~\Cref{cmfwg:regret}, we know that 
$$\mathbb{E}\left[\sum_{t=1}^T f_t(\bmc) -\sum_{t=1}^T f_t(\bmo)\right]=O\left(\sqrt{T}\right).$$
Therefore, it suffices to prove the upper bound on the constraint violation.
Note that
\begin{align*}
&\mathbb{E}\left[\sum_{t=1}^T g_t(\bmc) \right]\\
&\leq \mathbb{E}\left[\sum_{t=1}^T\left(f_t(\bmc)+(\lambda^*+1)g_t(\bmc)\right) - \sum_{t=1}^T\left(f_t(\bmo)+\lambda_tg_t(\bmo)\right)\right]\\
&\leq C_1T^{\frac{1}{2}+\beta} + 24RD\sqrt{d}T^{\frac{1}{2}-\beta}\lambda^{*2}\\
&\leq O(\sqrt{T})
\end{align*}
where the first inequality is due to~\Cref{expectation2}, the second inequality is from~\Cref{lemma2:regret3}, and the last inequality holds because of~\eqref{alg2-parameters}.
\end{proof}

\end{document}